\documentclass[12pt,a4paper]{amsart}
\usepackage[T1]{fontenc}
\usepackage[top=35mm, bottom=35mm, left=30mm, right=30mm]{geometry}    
\usepackage{amssymb}
\usepackage{mathtools}
\usepackage{verbatim,enumerate}

\usepackage[looser]{newpxtext}
\usepackage[smallerops]{newpxmath}

\usepackage[hidelinks,pagebackref]{hyperref}
\renewcommand*{\backref}[1]{}
\renewcommand*{\backrefalt}[4]{%
\ifcase #1
Not cited%
\or
(Cited on page~#2)%
\else
(Cited on pages~#2)%
\fi
}

\theoremstyle{definition}
\newtheorem{defn}{Definition}[section]

\newtheorem{rem}[defn]{Remark}
\newtheorem{ques}[defn]{Question}

\theoremstyle{plain}
\newtheorem{thm}[defn]{Theorem}
\newtheorem{lem}[defn]{Lemma}
\newtheorem{prop}[defn]{Proposition}
\newtheorem{coro}[defn]{Corollary}

\newtheorem{remark}[defn]{Remark}

\title[Typical dynamical properties of operators on $\ell_p$]
{Typical dynamical properties of operators on $\ell_p$}

\author[J. Li]{Jian Li}
\address[J. Li]{Institute for  Mathematical Sciences and Artificial Intelligence \& Department of Mathematics,
	Shantou University, Shantou, 515821, Guangdong, China}
\email{lijian09@mail.ustc.edu.cn}
\urladdr{https://orcid.org/0000-0002-8724-3050}

\author[Q. Liao]{Qijing Liao}
\address[Q. Liao]{Department of Mathematics,
	Shantou University, Shantou, 515821, Guangdong, China}
\email{liaoqijing1@outlook.com}

	

\date{\today}

\begin{document}

\begin{abstract}
We investigate the typical dynamical properties of hypercyclic operators in $\mathcal{L}_M(X)$, the set of all bounded linear operators on $X$ whose norms are at most $M$, when $X=\ell_p$, $1< p<\infty$.
We show that, with respect to SOT$^*$, a typical operator $T\in \mathcal{L}_M(X)$ is weakly mixing, is weakly disjoint from a given hypercyclic operator $S$, is not topologically ergodic, and satisfies $(T,T^2,\dotsc,T^k)$ is disjoint hypercyclic for any $k\geq 2$.
We also show that the similar typical dynamical properties for the concrete family $\mathcal{M}=\{I+B_w\in \mathcal{L}(X)\colon w\in c_0(\mathbb{Z})\}$, endowed with the norm topology, 
where $B_w$ is a bilateral weighted backward shift.
\end{abstract}

\maketitle 

%\tableofcontents

\section{Introduction}

Linear dynamics has grown into a very active research field over the past few decades. 
We refer to the monographs \cite{BM09} and \cite{G11} for a comprehensive study of the dynamics of linear operators.
A central notion in this theory is hypercyclicity.
A continuous linear operator on an infinite-dimensional Fr\'echet space over $\mathbb{R}$ or $\mathbb{C}$ is called hypercyclic if there exists a vector whose orbit under the operator is dense in the space.
Ansari \cite{A97} and Bernal \cite{B99} independently established the existence of hypercyclic operators on every infinite-dimensional separable Banach space, and Bonet and Peris \cite{BP98} later extended this result to Fr\'echet spaces.

The existence of hypercyclic operators on a given space naturally leads to the study of the size of the collection of hypercyclic operators.
Wu observed in \cite{W94} that the collection of hypercyclic operators is nowhere dense with respect to the norm topology. 
However, Chan in \cite{C02} proved that hypercyclic operators form an SOT-dense subset of the space of all operators on every infinite-dimensional separable Hilbert space; subsequently, B\`es and Chan in \cite{BC03} extended this result to Fr\'echet spaces.

Given a separable Banach space $X$, we are interested in how abundant the hypercyclic operators are among all bounded linear operators on $X$.
A natural way to quantify this richness is via the notion of typicality in the sense of Baire-category. 
More precisely, let $\mathcal{X}$ be a Baire space and $\Phi$ be a property of its elements.
We say that $\Phi$ holds typically in $\mathcal{X}$ if the set of points satisfying $\Phi$ is comeager (or residual) in $\mathcal{X}$, or equivalently, if it contains a dense $G_\delta$ subset of $\mathcal{X}$.
As mentioned above, hypercyclic operators are far from typical in the norm topology.
This motivates the study of typical behavior with respect to the strong operator topology or the strong-star operator topology, denoted by SOT and SOT$^*$, respectively.
For every separable Banach space $X$ and every $M>0$, 
denote by $\mathcal{L}(X)$ the space of bounded linear operators on $X$.
Then $\mathcal{L}_M(X)$, that is, the set of all $T\in \mathcal{L}(X)$ whose norm is less than or equal to $M$, is a Polish space with respect to SOT. 
Moreover, when $X^*$ is also separable, $\mathcal{L}_M(X)$ is a Polish space with respect to SOT$^*$, see e.g. \cite{P89}.

Although there have been some initial studies on typical properties of operators, they have mainly focused on typical properties of contractions.
In the context of Hilbert space $H$, Eisner showed in \cite{E10} that a typical contraction is unitary for the weak operator topology.
Eisner and Mátrai \cite{EM13} further studied typical contractions 
with respect to SOT, and showed that a typical element of $(\mathcal{L}_1(H),SOT)$ is unitarily equivalent to the operator $B^{(\infty)}$, the countable direct $\ell_2$-sum of the unilateral backward shift $B$ on $\ell_2(\mathbb{N})$.
Hence the SOT gives a completely and rather rigid description of typical operators on Hilbert space, while the SOT$^*$ setting is more delicate and leads to more meaningful typical questions.
Therefore, we always use SOT$^*$ rather than SOT.
More recently, in \cite{GMMlp21,GM22,GMM26}, the authors investigated typical contractions on $\ell_p$-spaces and $c_0$.
In the monograph \cite{GMM21}, Grivaux et al.\@ primarily studied the typical dynamical properties of operators: 
an SOT$^*$-typical hypercyclic operator is weakly mixing but not mixing, has no eigenvalues and admits no non-trivial invariant measure, yet is densely distributionally chaotic, etc.

Building on these results, it is natural to ask whether similar typicality phenomena persist beyond the Hilbert space setting.
More precisely, given a separable infinite-dimensional Banach space $X$ and $M>1$, we wish to determine which dynamical properties hold densely, or even typically, in the closed operator ball $\mathcal{L}_M(X)$ when it is endowed with SOT or SOT$^*$.
However, this appears to be difficult to achieve. 
We have therefore restricted our study to classical sequence spaces, primarily $\ell_p$ spaces.
We go beyond hypercyclicity itself to investigate the typicality of several stronger dynamical properties, including weak mixing, weak disjointness and various forms of disjoint hypercyclicity.
The motivation for studying disjoint hypercyclicity stems from the similarity principle in \cite{S10}, which connects hypercyclicity of direct sums with disjoint hypercyclicity.

The following is the first main result of this paper.

\begin{thm}\label{main-result-1}
Let $X=\ell_p$, $1<p<\infty$ and $M>1$.
Denote $\mathcal{L}_M(X):=\{T\in \mathcal{L}(X)\colon \|T\|\leq M\}$.
Then a typical operator $T\in \mathcal{L}_M(X)$ with respect to SOT$^*$ has the following properties:
\begin{enumerate}
    \item $T$ is weakly mixing;
    \item $T$ is not topologically ergodic;
    \item $T$ is weakly disjoint from a given hypercyclic operator $S$ on a Banach space $Y$;
    \item $(T,T^2,\dotsc,T^k)$ is disjoint hypercyclic for any $k\geq 2$.
\end{enumerate}
\end{thm} 

\begin{remark}
    Let $X=\ell_p$, $1<p<\infty$, and let $q$ be the conjugate exponent of $p$. 
    Since $X^*=\ell_q$ and $X$ is reflexive, the adjoint map 
    \[
        \Phi:\mathcal{L}_M(X)\to \mathcal{L}_M(X^*),\quad \Phi(T)=T^*
    \]
    is a homeomorphism when both operator balls are endowed with SOT$^*$. 
    Therefore, any residual result established for operators on $\ell_q$ can be transferred to a corresponding residual result for the adjoints of operators on $\ell_p$.
\end{remark} 

Weighted shift perturbations of the identity have also attracted considerable interest in linear dynamics.
In \cite{S95}, Salas proved that every operator of the form $I+B_w$, where $B_w$ is a unilateral weighted backward shift with a positive weight $w$, is hypercyclic on $\ell_p(\mathbb{N}_0)$. 
However, the bilateral case is much more complicated.
Rodr\'iguez-Mart\'inez in \cite[Proposition 2.4]{RA12} showed that, for a dense $G_\delta$ set of weights $w\in c_0(\mathbb{Z})$, both $I+B_w$ and $I+B_w^*$ are hypercyclic on $\ell_2(\mathbb{Z})$.
Related existence results were obtained by Shkarin in \cite[Corollary 6.13]{S15}.
Nevertheless, these results do not characterize the hypercyclicity of $I+B_w$ in the bilateral setting. Shkarin raised this question in \cite[Problem 13.25]{S15}.
On the other hand, \cite[Corollary 13.24]{S15} shows that symmetric weights do not yield hypercyclic operators of the form $I+B_w$ on $\ell_2(\mathbb{Z})$.

This motivates our study of typical dynamical properties for the concrete family $\mathcal{M}=\{I+B_w\in \mathcal{L}(X)\colon w\in c_0(\mathbb{Z})\}$, endowed with the norm topology.
Since this family is naturally parametrized by $c_0(\mathbb{Z})$ and the map $w\to I+B_w$ identifies the operator norm with the sup norm on weights, we study typicality in $\mathcal{M}$ with respect to the norm topology.

The second main result of this paper is as follows.

\begin{thm}\label{main-result-2}
Let $X=\ell_p(\mathbb{Z})$, $1<p<\infty$ and $\mathcal{M}=\{I+B_w\colon w\in c_0(\mathbb{Z})\}\subset (\mathcal{L} (X),\|\cdot\|)$. 
Then a typical operator $I+B_w\in \mathcal{M}$ with respect to the norm topology has the following properties:
\begin{enumerate}
    \item $I+B_w$ is weakly mixing;
    \item $I+B_w$ is not topologically ergodic;
    \item $I+B_w$ is weakly disjoint from a given hypercyclic operator $S$ on a Banach space $Y$.
    \item  $(I+B_w,(I+B_w)^2,\dotsc,(I+B_w)^k)$ is disjoint hypercyclic for any $k\geq 2$.
\end{enumerate}
\end{thm} 

The paper is organized as follows. 
Section 2 is devoted to preliminary material, where we introduce the notations and conventions used throughout the paper.
In Section 3, we prove assertions (1) and (2) of Theorem \ref{main-result-1}.
We investigate the Baire-category size of several natural classes of operators, including weakly mixing, syndetically transitive, mixing, and mildly mixing operators. 
We also consider prescribed vector versions, where a fixed non-zero vector $x\in X$ is required to be hypercyclic.
In Section 4, we prove assertions (3) and (4) of Theorem \ref{main-result-1}.
We mainly study weak disjointness and disjoint hypercyclicity via direct sum.
We prove typical results both for weakly disjoint partners of a fixed hypercyclic operator and weakly disjoint pairs.
We then turn to the diagonal setting, obtaining generic results on disjoint hypercyclicity and its prescribed-vector variant, in which the fixed vector $(x,x)$ is required to be hypercyclic with respect to $T\oplus S$.
In Section 5, we establish typicality results for the concrete family $\mathcal{M}=\{I+B_w\in \mathcal{L}(X) \colon w\in c_0(\mathbb{Z})\}$ of weighted shift perturbations of the identity, endowed with the norm topology.
In particular, we prove Theorem \ref{main-result-2}.

\section{Preliminaries}

In this section we recall some definitions and results which will be used later. 
By $\mathbb{N}$ and $\mathbb{N}_0$ we denote the set of positive integers and non-negative integers.

\subsection{Operators on Banach spaces}
In this subsection we recall some notions and results on operators on Banach spaces that will be used in this paper. 
We refer the reader to e.g. \cite{P89} for more details.

Throughout this paper, $\mathbb{K}$ denotes either the real field $\mathbb{R}$ or the complex field  $\mathbb{C}$.
Let $X$ be a Banach space over $\mathbb{K}$.
Denote the space of bounded linear operators on $X$ by $\mathcal{L}(X)$.
For $T\in \mathcal{L}(X)$, we denote its operator norm by $\|T\|$.
The \emph{dual space} $X^*$ of $X$ is the space of continuous linear functionals on $X$, which is a Banach space with respect to the norm topology.
The norm topology is induced by the norm $\|T\|=\sup_{\|x\|\leq 1}\|Tx\|$.
For $T\in \mathcal{L}(X)$, the \emph{adjoint} $T^*$ of $T$ is defined by $\langle x,T^*f \rangle=\langle Tx,f\rangle$ for all $x\in X$ and $f\in X^*$.
Moreover, $\|T^*\|=\|T\|$. 

\begin{defn}
    Let $X$ be a Banach space. The \emph{strong operator topology} (or SOT for short) on $\mathcal{L}(X)$ is defined as follows: any $T\in \mathcal{L}(X)$ has a neighborhood basis consisting of sets of the form 
    \[
        U_{T;x_1,\dotsc, x_n;\epsilon}=\{S\in \mathcal{L}(X)\colon \|(S-T)x_k\|<\epsilon \quad \text{for} \ k=1,\dotsc,n \},
    \]
    where $x_1,\dotsc, x_n\in X$.
    The SOT is the topology of pointwise convergence of operators. 
    Thus, a sequence (or more generally a net) $(T_i)_i\subset \mathcal{L}(X)$ is SOT-convergent to $T\in \mathcal{L}(X)$ if and only if $T_i x\to Tx$ in norm for every $x\in X$.

    The \emph{strong$^*$ operator topology} (or SOT$^*$ for short), which is the self-adjoint version of SOT, on $\mathcal{L}(X)$ is defined as follows: a basis of SOT$^*$-neighborhoods of $T\in \mathcal{L}(X)$ is provided by the sets of the form
    \[
        U_{T;x_1,\dotsc, x_n;y_1,\dotsc,y_m;\epsilon}=\left \{
        S\in \mathcal{L}(X)\ \middle|\ 
        \begin{aligned}
            \|(S-T)x_i\|<\epsilon, \quad i=1,\dotsc,n \ \\
            \|(S-T)^* y_j\|<\epsilon, \quad j=1,\dotsc,m
        \end{aligned}
        \right \},
    \]
    where $x_1,\dotsc, x_n\in X$, $y_1,\dotsc, y_m\in X^*$, $n,m\geq 1$, and $\epsilon>0$.
    Thus, a net $(T_i)_i\subset \mathcal{L}(X)$ is SOT$^*$-convergent to $T\in \mathcal{L}(X)$ if and only if $T_i \to T$ and $T_i^* \to T^*$ with respect to SOT.
\end{defn}

Among these topologies, the SOT is the coarsest, the norm topology is the finest, and the SOT$^*$ lies strictly in between. 

On the whole space $\mathcal{L}(X)$, the SOT and SOT$^*$ generally fail to provide a suitable Baire-category setting.
In general, it is neither metrizable nor second countable.
Therefore, we restrict our attention to the bounded operator ball $\mathcal{L}_M(X)$.
For any $M>0$, denote
\[
\mathcal{L}_M(X):=\{T\in \mathcal{L}(X):\|T\|\leq M\}.
\]

We next show that every closed ball $\mathcal{L}_M(X)$ is a Polish space (i.e. separable and completely metrizable space) with respect to these topologies.

\begin{prop}[{\cite[Section 4.6.2]{P89}}]
Let $X$ be a separable Banach space with a separable dual.
Then for every $M>0$, $\mathcal{L}_M(X)$ is a Polish space with respect to both SOT and SOT$^*$.
\end{prop}

We shall mainly work with the sequence spaces $\ell_p$, $1\leq p<\infty$ and $c_0$.

\begin{defn}
    For $1\leq p<\infty$, the unilateral sequence space $\ell_p(\mathbb{N}_0)$ is defined by 
    \[
        \ell_p(\mathbb{N}_0):=\biggl\{ x=(x_n)_{n\geq 0}\colon \sum_{n=0}^\infty |x_n|^p<\infty \biggr \},
    \]
    equipped with the norm $\|x\|_p=(\sum_{n=0}^\infty |x_n|^p)^{\frac{1}{p}}$.
    The bilateral sequence space $\ell_p(\mathbb{Z})$ is defined by 
    \[
        \ell_p(\mathbb{Z}):=\biggl\{ x=(x_n)_{n\in \mathbb{Z}}\colon \sum_{n\in \mathbb{Z}} |x_n|^p<\infty \biggr \},
    \]
    equipped with the norm $\|x\|_p=(\sum_{n\in \mathbb{Z}} |x_n|^p)^{\frac{1}{p}}$.
    The unilateral sequence space $c_0(\mathbb{N}_0)$ is defined by 
    \[
        c_0(\mathbb{N}_0):=\left\{ x=(x_n)_{n\geq 0}\colon  x_n\to 0 \text{ as } n\to \infty \right \},
    \]
    equipped with the norm $\|x\|_\infty=\sup_{n\geq 0}|x_n|$.
\end{defn}
Similarly one can define the bilateral sequence space $c_0(\mathbb{Z})$.
Throughout this paper, when we write just $\ell_p$ (or $c_0$), it defaults to $\ell_p(\mathbb{N}_0)$ (or $c_0(\mathbb{N}_0)$).
The canonical basis of $\ell_p(\mathbb{N}_0)$ and $c_0(\mathbb{N}_0)$ is denoted by $(e_n)_{n\geq 0}$, that is, $e_n=(\delta_{kn})_{k\geq 0}$.
We denote by $c_{00}$ the space of finitely supported sequences.
For $x\in c_{00}$, denote $\operatorname{supp}(x)=\{i\colon x_i\neq 0\}$.
The subspace $c_{00}$ is dense in $\ell_p$.
It is well known that 
$(\ell_p)^*\cong \ell_q$ $(\frac{1}{p}+\frac{1}{q}=1)$, $(\ell^1)^*\cong \ell^\infty$, and $(c_0)^*\cong \ell^1$.

\subsection{Linear dynamics}
In recent years, research on linear dynamical systems has attracted increasing attention. For further details, see \cite{BM09} and \cite{G11}.

Let $X$ be a Banach space and $T \colon X \to X$ be a continuous linear operator.
We say that the pair  $(X,T)$ is  a \emph{linear dynamical system}.
For brevity, we denote it by the operator $T$, or the map $T \colon X \to X$.

\begin{defn}
    An operator $T\colon X\to X$ is said to be \emph{hypercyclic} if there is some $x\in X$ such that the orbit of $x$ under $T$, that is, $\{T^nx \colon n\geq 0\}$ is dense in $X$. Such an $x$ is called a \emph{hypercyclic vector} for $T$. We denote the set of hypercyclic vectors for $T$ by $\operatorname{HC}(T)$.
\end{defn}

For each pair of non-empty open subsets $U$ and $V$ of $X$, we recall that the return time set is defined by 
\[
    N_T(U,V)=N(U,V):=\{n\in\mathbb{N}_0 \colon  T^n(U)\cap V\neq\emptyset \}.
\]
For every $x\in X$ and $U\subset X$, we denote
\[
    N_T(x,U)=N(x,U):=\{n\in\mathbb{N}_0 \colon  T^nx\in U \}.
\]
An operator $T$ is said to be \emph{topologically transitive} if $N_T(U,V)\neq\emptyset$.
$T$ is \emph{strongly mixing} (or \emph{mixing}) if all sets $N_T(U,V)$ are cofinite rather than just non-empty, and \emph{weakly mixing} if the product system $(X\times X,T\times T)$ is topologically transitive. 
Equivalently, if $N_T(U_1,V_1)\cap N_T(U_2,V_2)\neq \emptyset$ for any non-empty open sets $U_1,V_1,U_2,V_2$ of $X$.
We set 
\[
    \operatorname{HC}(X):=\{T\in \mathcal{L}(X) \colon T \text{ is hypercyclic}\},
\]
\[
    \operatorname{MIX}(X):=\{T\in \mathcal{L}(X) \colon T \text{ is mixing}\} 
\]
and 
\[
    \operatorname{WMIX}(X):=\{T\in \mathcal{L}(X) \colon T \text{ is weakly mixing}\}
\]
Thus, 
\[
    \operatorname{MIX}(X)\subset \operatorname{WMIX}(X) \subset \operatorname{HC}(X).
\]

By the Birkhoff transitivity theorem (see e.g.\@ \cite[Theorem 1.16]{G11}), an operator $T\colon X\to X$ is hypercyclic if and only if it is topologically transitive. 
Moreover, whenever these equivalent conditions hold, the set $\operatorname{HC}(T)$ of hypercyclic vectors  is a dense $G_\delta$ subset of $X$.

The notion of weak disjointness was introduced by Peleg in topological dynamics \cite{P72}. 
In linear dynamics, it is equivalent to the hypercyclicity of direct sums, a topic studied in \cite{B07,BP07,BMPP19,C24,S95}. 
Recently, in \cite{LLR25}, Li et al.\@ introduced the concept of weak disjointness within the framework of linear dynamical systems.

\begin{defn} %[{\cite[Definition 3.1]{LLR25}}]
Given two operators $T \colon X\to X$ and $S \colon Y\to Y$, we call them \emph{weakly disjoint} if their direct sum $T\oplus S$ is hypercyclic.
\end{defn}

\begin{defn}
     A subset $F$ of $\mathbb{N}_0$ is called \emph{syndetic} if $F$ has bounded gaps.
     A subset $F$ of $\mathbb{N}_0$ is called \emph{thickly syndetic} if for every $k$ there exists a syndetic set $S$ such that $S+\{0,\dotsc,k\}\subset F$.

    An operator $T\colon X\to X$ is said to be \emph{topologically ergodic} (or \emph{syndetically transitive}) provided that $N(U,V)$ is syndetic  for each pair of non-empty open subsets $U$ and $V$ of $X$.
    An operator $T\colon X\to X$ is said to be \emph{thickly syndetic transitive} provided that $N(U,V)$ is thickly syndetic for each pair of non-empty open subsets $U$ and $V$ of $X$.
\end{defn}

In the setting of linear dynamics, topologically ergodic operators are weakly mixing \cite[Theorem 6.31]{BM09}. 
If $T$ is topologically ergodic, then $T$ is thickly syndetic transitive \cite[Lemma 2.3]{BMPP19}.

In 2024, Cardeccia characterized in \cite{C24} the operators which are weakly disjoint from all weakly mixing operators.

\begin{thm}[{\cite[Theorem 5.3]{C24}}]\label{treg-wd-weak-mixing}
Let $T\colon X\to X$ be an operator. 
The following assertions are equivalent: 
    \begin{enumerate}
        \item $T$ is topologically ergodic; 
         \item $T$ is weakly disjoint from all weakly mixing operators; 
        \item $T\oplus S$ is weakly mixing for each weakly mixing operator $S\colon Y\to Y$. 
    \end{enumerate}
\end{thm}

Moothathu showed in \cite{M13} that a topologically ergodic operator cannot have a hypercyclic adjoint.

\begin{thm}[{\cite[Theorem 2]{M13}}]\label{T-TERG-T*-not-hyper}
    Let $T\colon X\to X$ be an operator. 
    If $T$ is syndetically transitive, then $T^*$ is not hypercyclic.
\end{thm}

Reiterative hypercyclicity was introduced by B\`es, Menet, Peris and Puig in their study of recurrence properties of hypercyclic operators \cite{BMPP16}.
\begin{defn}
    An operator $T\colon X\to X$ is said to be \emph{reiteratively hypercyclic} if there exists a vector $x\in X$ such that, for every non-empty open set $U\subset X$, the return set $N_T(x,U)$ has positive upper Banach density.
\end{defn}

In the same paper \cite{BMPP16}, they proved that reiterative hypercyclicity implies topological ergodicity.

\begin{prop}[{\cite[Proposition 25]{BMPP16}}]\label{T-reit-T-terg}
    Let $X$ be a separable $F$-space and $T\in \mathcal{L}(X)$. 
    If $T$ is reiteratively hypercyclic, then $T$ is topologically ergodic.
\end{prop}

The notion of mild mixing was first introduced in ergodic theory by Furstenberg and Weiss \cite{FW06}. 
Motivated by this ergodic‑theoretic concept, Glasner \cite{G04} and, independently, Huang and Ye \cite{HY04} later introduced a topological analogue.
Later, this notion was formally introduced into the framework of linear dynamical systems in \cite{LLR25}.

\begin{defn} 
An operator $T\colon X\to X$ is said to be \emph{mildly mixing} provided that it is weakly disjoint from all hypercyclic operators. 
\end{defn}

B\`es and Peris in \cite{BP07} and Bernal-Gonz\'alez in \cite{B07} independently introduced the notion of d-hypercyclicity (i.e. disjoint hypercyclicity).

\begin{defn}
    Let $X$ be a topological vector space, $m\geq 2$ and $\mathbb{T}=(T_1,\dotsc,T_m)\in \mathcal{L}(X)^m$.  
    The $m$-tuple $\mathbb{T}$ is called \emph{d-hypercyclic} if there exists $x\in X$ such that 
    \[
        \{(T_1^nx,\dotsc,T_m^nx) \colon n\geq 0\}
    \]
    is dense in $X^m$.
    We denote by $\operatorname{dHC}(T_1,\dotsc,T_m)$ the set of all vectors $x\in X$ such that $\{(T_1^nx,\dotsc,T_m^nx) \colon n\geq 0\}$ is dense in $X^m$.
    
    The $m$-tuple $\mathbb{T}$ is called \emph{densely d-hypercyclic} if $\operatorname{dHC}(T_1,\dotsc,T_m)$ is dense in $X$.
    Clearly, dense d-hypercyclicity implies d-hypercyclicity.
\end{defn}

\begin{defn}
    Let $X=\ell_p(\mathbb{Z})$, $1\leq p<\infty$, or $X=c_0(\mathbb{Z})$, and $w=(w_n)_{n\in \mathbb{Z}}$ be a bounded scalar sequence, called a weight sequence.
    The weighted backward shift $B_w \colon X\to X$ is defined by 
    \[
        B_w(x_n)_{n\in \mathbb{Z}}=(w_{n+1}x_{n+1})_{n\in \mathbb{Z}},
    \]
    for every $x=(x_n)_{n\in \mathbb{Z}}\in X$.
\end{defn}
Similarly one can define the unilateral weighted backward shift on spaces $X=\ell_p(\mathbb{N}_0)$, $1\leq p<\infty$, or $X=c_0(\mathbb{N}_0)$.

\section{Typical dynamical properties of operators on  \texorpdfstring{$\ell_p$}{lp}}

In this section, we show that when $X=\ell_p$, $1<p<\infty$ and $M>1$, a typical operator in $\mathcal{L}_M(X)$ with respect to SOT$^*$ is weakly mixing but not topologically ergodic. 
Additionally, we determine the Borel complexity of the set of topologically ergodic operators and show that this set is Borel but meager.
Moreover, given a fixed non-zero vector $x\in X$, we study the typical questions of those operators for which $x$ is hypercyclic.
Recall that Grivaux et al.\@ established in \cite{GMM21} that, on $\ell^2$ and $M>1$, a typical operator in $\mathcal{L}_M(\ell^2)$ with respect to SOT$^*$ is weakly mixing yet not mixing.
We are motivated by this result.

Grivaux et al.\@ provided a criterion in \cite[Lemma 2.6]{GMM21} for the density of a class of operators in the unit ball of bounded operators with respect to SOT (or SOT$^*$). 
An analogous result holds for the spaces $\ell_p$ and $c_0$.

\begin{lem}\label{dense criteria-lp}
    Let $X=\ell_p$, $1\leq p<\infty$, and $\Gamma(X)\subset \mathcal{L}_M(X)$ for some $M>0$. 
    Let $(e_k)_{k\geq 0}$ be the canonical basis of $X$, and for each $r\geq 1$, denote $X_r=\operatorname{span}\{e_0,\dotsc,e_{r-1}\}$, $P_r(x_0,x_1,\dotsc)=(x_0,\dotsc,x_{r-1},0,\dotsc)$. 
    Suppose that for each $r\geq 1$, $A\in \mathcal{L}(X_r)\cap \mathcal{L}_M(X)$ and every $\epsilon>0$, there exists  $T\in \Gamma(X)$ such that 
    \[
        \|(T-A)e_k\|_p<\epsilon, \quad k=0,\dotsc,r-1.
    \]
    Then $\Gamma(X)$ is dense in $(\mathcal{L}_M(X),SOT)$.

    Moreover, if $1<p<\infty$ and $\frac{1}{p}+\frac{1}{q}=1$, and if one can choose $T\in \Gamma(X)$ such that in addition
    \[
        \|(T^*-A^*)e_k\|_q<\epsilon, \quad k=0,\dotsc,r-1.
    \]
    Then $\Gamma(X)$ is dense in $(\mathcal{L}_M(X),SOT^*)$.
\end{lem}

\begin{lem}\label{dense criteria-c0}
    Let $X=c_0$ and $\Gamma(X)\subset \mathcal{L}_M(X)$ for some $M>0$. 
    Let $(e_k)_{k\geq 0}$ be the canonical basis of $X$, and for each $r\geq 1$, denote $X_r=\operatorname{span}\{e_0,\dotsc,e_{r-1}\}$, $P_r(x_0,x_1,\dotsc)=(x_0,\dotsc,x_{r-1},0,\dotsc)$. 
    Suppose that for each $r\geq 1$, $A\in \mathcal{L}(X_r)\cap \mathcal{L}_M(X)$ and every $\epsilon>0$, there exists  $T\in \Gamma(X)$ such that 
    \[
        \|(T-A)e_k\|_\infty<\epsilon, \quad  
        \|(T^*-A^*)e_k\|_1<\epsilon, \quad   k=0,\dotsc,r-1.
    \]
    Then $\Gamma(X)$ is dense in $(\mathcal{L}_M(X),SOT^*)$.
\end{lem}

Grivaux et al.\@ proved in \cite[Corollary 2.12]{GMM21} that $G$-$\operatorname{MIX}_M(\ell_2)\cap \operatorname{CH}_M(\ell_2)$ is dense in $(\mathcal{L}_M(\ell_2),SOT^*)$.
We use the definition of mixing directly, without relying on $G$-$\operatorname{MIX}_M(X)$, to simplify the proof for determining the density of the collection of mixing operators.

\begin{prop}\label{MIX-dense-l^p}
Let $X=\ell_p$, $1<p<\infty$ and $M>1$. 
Then 
\[
\operatorname{MIX}_M(X):=\{T\in \mathcal{L}_M(X) \colon T \text{ is mixing}\}
\]
is dense in $(\mathcal{L}_M(X),SOT^*)$.
\end{prop}

\begin{proof}
Fix $r\geq 1$. 
Set $X_r=\operatorname{span}\{e_0,\dotsc,e_{r-1}\}$, and let $A\in \mathcal{L}(X_r)\cap \mathcal{L}_M(X)$. 
Without loss of generality, we can assume that $\|A\|<M$. 
Let $\epsilon>0$.
By Lemma \ref{dense criteria-lp}, it is enough to construct $T\in \operatorname{MIX}_M(X)$ such that $T|_{X_r}=A$ and $\|(T^*-A^*)e_k\|_q<\epsilon$ for $k=0,\dotsc,r-1$.

We identify $\ell_p$ with the block sequence space
    \[
        \ell_p(\mathbb{N}_0;\mathbb{K}^r)=\biggl\{x=(x_0,x_1,x_2,\dotsc) \colon x_j\in \mathbb{K}^r,\sum_{j\geq 0}\|x_j\|_p^p<\infty
        \biggr\}.
    \]
    The first block $\mathbb{K}^r$ is identified with $X_r$. 
    Let $q=\frac{p}{p-1}$. 
    Choose $w>0$ with $\max(1,\|A\|)<w<M$.
    Since $\|A\|<M$, choose $\delta>0$ so small that $\delta<\epsilon$ and $\|A\|^q+\delta^q<M^q$.
    Next we define an operator $T=B_{A,w}\in \mathcal{L}(\ell_p)$ by 
    \[
        B_{A,w}(x_0,x_1,x_2,\dotsc)=(Ax_0+\delta x_1,wx_2,wx_3,\dotsc).
    \]
    Then $\|B_{A,w}\|\leq M$.
    Indeed, for each $x=(x_j)_{j\geq 0}\in X$, by H\"older's inequality,
    \begin{align*}
        \|Ax_0+\delta x_1\|_p&\leq \|A\|\cdot \|x_0\|_p+\delta \|x_1\|_p \\
        &\leq (\|A\|^q+\delta^q)^{\frac{1}{q}}(\|x_0\|_p^p+\|x_1\|_p^p)^{\frac{1}{p}} 
        \leq M(\|x_0\|_p^p+\|x_1\|_p^p)^{\frac{1}{p}}.
    \end{align*}
    Hence 
    \begin{align*}
        \|B_{A,w}x\|_p^p&=\|Ax_0+\delta x_1\|_p^p +\sum_{j\geq 1}\|wx_{j+1}\|_p^p \\
        &\leq M^p(\|x_0\|_p^p+\|x_1\|_p^p)+w^p\sum_{j\geq 1}\|x_{j+1}\|_p^p 
        \leq M^p \sum_{j\geq 0}\|x_j\|_p^p.
    \end{align*}
    Thus $\|B_{A,w}\|\leq M$, so $B_{A,w}\in \mathcal{L}_M(\ell_p)$. 
    By the construction of $B_{A,w}$, its restriction to the first block is precisely $A$.
    Moreover, we may get $\|(B_{A,w}-A)^*e_k\|<\epsilon$ for $0\leq k\leq r-1$.

    We next show that $B_{A,w}$ is mixing. 
    Let $U$ and $V$ be non-empty open sets of $X$.
    Since the set of finitely supported block vectors is dense in $\ell_p(\mathbb{N}_0;\mathbb{K}^r)$, we may choose $u=(u_0,\dotsc,u_s,0,0,\dotsc)\in U$ and $v=(v_0,\dotsc,v_s,0,0,\dotsc)\in V$. 
    For each $N>s$, define a vector $z^{(N)}=(z^{(N)}_0,z^{(N)}_1,\dotsc)\in \ell_p$ as follows
    \[
        z^{(N)}_k=u_k\ (0\leq k\leq s), \quad  
        z^{(N)}_k=0\ (s<k<N),
    \]
    \[
        z^{(N)}_N=\delta^{-1}w^{-(N-1)}\Bigl(v_0-A^Nu_0-\delta\sum_{k=1}^sA^{N-k}w^{k-1}u_k\Bigr),
    \]
    \[
         z^{(N)}_{N+j}=w^{-N}v_j\ (1\leq j \leq s), \quad \text{all remaining blocks equal to } 0.
    \]
    We claim that $B_{A,w}^N z^{(N)}=v$. 
    For each $1\leq j \leq s$, $(B_{A,w}^Nz^{(N)})_j=w^N z^{(N)}_{N+j}=v_j$ and 
    \begin{align*}
        (B_{A,w}^Nz^{(N)})_0 &=A^N z^{(N)}_0+\delta \sum_{k=1}^N A^{N-k} w^{k-1} z^{(N)}_k \\
        &=A^N u_0+\delta \sum_{k=1}^sA^{N-k}w^{k-1} u_k+ \delta w^{N-1} z^{(N)}_N=v_0. 
    \end{align*}
    Thus $B_{A,w}^N z^{(N)}=v$. 
    Next we prove that $z^{(N)}\to u$ in $X$. 
    Since $z^{(N)}_k=u_k (0\leq k\leq s)$, it is enough to show that all additional tail blocks tend to $0$. 
    For the block $z^{(N)}_N$, 
    \[
        \|z^{(N)}_N\|_p \leq \delta^{-1}w^{-(N-1)}\|v_0\|_p+\delta^{-1}w\Bigl(\frac{\|A\|}{w}\Bigr)^N\|u_0\|_p +\sum_{k=1}^s \Bigl(\frac{\|A\|}{w}\Bigr)^{N-k}\|u_k\|_p.
    \]
    Since $w>1$ and $\|A\|<w$, $\|z^{(N)}_N\|_p\to 0$ as $N\to \infty$. 
    For $1\leq j \leq s$, 
    \[
        \|z^{(N)}_{N+j}\|_p=w^{-N}\|v_j\|_p \to 0 \quad (N\to \infty).
    \]
    Therefore $z^{(N)}\to u$ in $X$. 
    For each sufficiently large $N$, we have that $z^{(N)}\in U$, while $B_{A,w}^N z^{(N)}=v\in V$. 
    Hence $B_{A,w}^N (U)\cap V\neq \emptyset$ for all sufficiently large $N$. 
    This proves that $B_{A,w}$ is mixing. 
    Therefore $\operatorname{MIX}_M(X)$ is dense in $(\mathcal{L}_M(X),SOT^*)$.
\end{proof}

However, density alone should not be taken as evidence that mixing is typical.
As we will see below, the class of mixing operators remains meager.
Grivaux et al.\@ proved in \cite[Proposition 2.16]{GMM21} that $\operatorname{WMIX}_M(\ell_2)$ is dense in $(\mathcal{L}_M(\ell_2),SOT^*)$.
We obtain a similar result for the spaces  $X=\ell_p$, $1<p<\infty$ and $X=c_0$.

\begin{prop}\label{WMIX-dense-l^p}
    Let $X=\ell_p$, $1<p<\infty$ and $M>1$. Then 
    \[
        \operatorname{WMIX}_M(X):=\{T\in \mathcal{L}_M(X) \colon T \text{ is weakly mixing}\}
    \]
    is residual in $(\mathcal{L}_M(X),SOT^*)$.
\end{prop}

\begin{proof}
    We first claim that $\operatorname{HC}_M(X)$ is a $G_\delta$ subset of $(\mathcal{L}_M(X),SOT^*)$.
    Choose a countable basis of non-empty open sets $(U_m)_{m\geq 1}$ of $X$. 
    Thus 
    \[
        \operatorname{HC}_M(X)=\bigcap_{s,t\geq 1}\bigcup_{n\geq 1}\{T\in \mathcal{L}_M(X) \colon  T^n(U_s)\cap U_t \neq \emptyset \}.
    \]
    Let $D_{n,s,t}=\{T\in \mathcal{L}_M(X) \colon  T^n(U_s)\cap U_t \neq \emptyset \}$. Our goal is to show that $D_{n,s,t}$ is SOT-open. Then $D_{n,s,t}$ is SOT$^*$-open. 
    
    Fix $T\in D_{n,s,t}$. 
    We need to find an SOT-neighborhood $U$ of $T$ such that $U\subset D_{n,s,t}$. 
    There exists $x\in U_s$ such that $T^nx\in U_t$. 
    Since $U_t$ is an open set, there exists $\epsilon>0$ such that $B(T^nx,\epsilon)\subset U_t$. 
    We observe that 
    \[
        \|S^nx-T^nx\|\leq \sum_{k=0}^{n-1}\|S^{n-1-k}\|\cdot \|(S-T)T^k x\|.
    \]
    Choose $U=\{T;y_0,\dotsc,y_{n-1};\delta\}$, where $y_k=T^kx$ for $k=0,\dotsc,n-1$, and $\delta=\frac{\epsilon}{nM^{n-1}}$. 
    Suppose $S\in U$. 
    Since $\|S\|\leq M$, 
    \[
         \|S^nx-T^nx\| \leq \sum_{k=0}^{n-1} M^{n-1-k} \|(S-T)y_k\|< \sum_{k=0}^{n-1} M^{n-1-k}\delta \leq nM^{n-1}\delta<\epsilon.
    \]
    Thus $S^nx\in U_t$. 
    Then $S\in D_{n,s,t}$. 
    So $D_{n,s,t}$ is an SOT-open set. 
    Therefore $\operatorname{HC}_M(X)$ is a $G_\delta$ subset of $\mathcal{L}_M(X)$. 
    Since $T$ is weakly mixing if and only if $T\oplus T$ is hypercyclic, and $T\to T\oplus T$ is SOT$^*$-continuous, $\operatorname{WMIX}_M(X)$ is a $G_\delta$ subset of $(\mathcal{L}_M(X),SOT^*)$.

    Next we show that 
    $\operatorname{WMIX}_M(X)$ is dense in $(\mathcal{L}_M(X),SOT^*)$. 
    By Proposition \ref{MIX-dense-l^p}, we have that $\operatorname{MIX}_M(X):=\{T\in \mathcal{L}_M(X) \colon T \text{ is mixing}\}$ is dense in $(\mathcal{L}_M(X),SOT^*)$.
    Since $\operatorname{MIX}_M(X)\subset \operatorname{WMIX}_M(X)$, $\operatorname{WMIX}_M(X)$ is dense in $(\mathcal{L}_M(X),SOT^*)$.
    Therefore $\operatorname{WMIX}_M(X)$ is residual in $(\mathcal{L}_M(X),SOT^*)$.
\end{proof}

We now consider another strengthening of topological transitivity based on bounded gaps.
In what follows, we show that in $\mathcal{L}_M(X)$, the set of topologically ergodic operators is meager but forms a Borel set.

\begin{thm}\label{syndetic-meager-l^p}
    Let $X=\ell_p$, $1<p<\infty$ and $M>1$. 
    Then 
    \[
    \operatorname{TERG}_M(X):=\{T\in \mathcal{L}_M(X) \colon  T \text{ is topologically ergodic}\}
    \] 
    is a Borel subset of $(\mathcal{L}_M(X), SOT^*)$, but it is meager in $(\mathcal{L}_M(X), SOT^*)$.
\end{thm}

\begin{proof}
    Since $X$ is separable, choose a countable basis of non-empty open subsets $(U_m)_{m\geq 1}$ of $X$.
    Recall that $T$ is syndetically transitive if and only if for any non-empty open sets $U$ and $V$ of $X$, there exists $L\geq 1$ such that for any $m\geq 0$, we can find $k\in \{m,\dotsc,m+L\}$ such that $T^k(U)\cap V\neq \emptyset$.
    For $r,s\geq 1$ and $k\geq 0$, define 
    \[
        F_{k,r,s}=\{T\in \mathcal{L}_M(X) \colon T^k(U_r)\cap U_s\neq \emptyset\}.
    \]
    It follows that 
    \[
        \operatorname{TERG}_M(X)=\bigcap_{r,s\geq 1}\bigcup_{L\geq 1}\bigcap_{m\geq 0}\bigcup_{k=m}^{m+L}F_{k,r,s}.
    \]
    Since $F_{k,r,s}$ is SOT-open, it is also SOT$^*$-open in $\mathcal{L}_M(X)$.
    Hence $\operatorname{TERG}_M(X)$ is a Borel subset of $(\mathcal{L}_M(X), SOT^*)$.
    
    It remains to show that $\operatorname{TERG}_M(X)$ is meager.
    By Theorem \ref{T-TERG-T*-not-hyper},
     \[
        \operatorname{TERG}_M(X) \subset \mathcal{L}_M(X)\setminus \{T \colon T^* \text{ is hypercyclic on }\ell_q\}.
    \]
    Since $X=\ell_p$ with $1<p<\infty$, we have $X^*=\ell_q$. 
    By Proposition \ref{WMIX-dense-l^p}, we have that $\{T\in \mathcal{L}_M(\ell_q) \colon T \text{ is hypercyclic on }\ell_q\}$ is a dense $G_\delta$ subset of $\mathcal{L}_M(\ell_q)$.
    Now consider the adjoint map
    \[
    \Phi \colon \mathcal{L}_M(X)\to \mathcal{L}_M(X^*), \quad \Phi(T)=T^*.
    \]
    Since $\Phi$ is a homeomorphism when both operator balls are endowed with SOT$^*$, the inverse image 
    \[
       \{T\in\mathcal{L}_M(X) \colon T^*\text{ is hypercyclic on }\ell_q\} 
    \]
    is a dense $G_\delta$ subset.
    Therefore $\mathcal{L}_M(X)\setminus \{T \colon T^* \text{ is hypercyclic on }\ell_q\}$ is meager.
    Consequently, $\operatorname{TERG}_M(X)$ is meager in $(\mathcal{L}_M(X), SOT^*)$.
\end{proof}

\begin{remark} 
    If $T$ is topologically ergodic, then $T$ is thickly syndetic transitive.
    Equivalently, we have that if $T^*$ is hypercyclic, then $T$ is not thickly syndetic transitive. 
\end{remark}

If $T$ is mildly mixing, then by definition $T$ is weakly disjoint from every hypercyclic operator.
In particular, $T$ is weakly disjoint from all weakly mixing operators.
By Theorem \ref{treg-wd-weak-mixing}, $T$ is topologically ergodic.
From this we obtain the following corollary.

\begin{coro}
    Let $X=\ell_p$, $1<p<\infty$ and $M>1$. Then 
    \[
        \operatorname{MMIX}_M(X):=\{T\in \mathcal{L}_M(X) \colon  T \text{ is mildly mixing}\}
    \]
    is meager in $(\mathcal{L}_M(X), SOT^*)$.
\end{coro}

\begin{coro}\label{MIX-meager-l^p}
    Let $X=\ell_p$, $1<p<\infty$ and $M>1$. Then 
    \[
        \operatorname{MIX}_M(X):=\{T\in \mathcal{L}_M(X) \colon  T \text{ is mixing}\}
    \]
    is meager in $(\mathcal{L}_M(X), SOT^*)$.
\end{coro}

It follows from Proposition \ref{T-reit-T-terg} that if $T$ is reiteratively hypercyclic, then $T$ is topologically ergodic, so we can also deduce the following.

\begin{coro}\label{reit-hyper-meager-l^p}
    Let $X=\ell_p$, $1<p<\infty$ and $M>1$. Then 
    \[
        \operatorname{REIT}_M(X):=\{T\in \mathcal{L}_M(X) \colon  T \text{ is reiteratively hypercyclic}\}
    \]
    is meager in $(\mathcal{L}_M(X), SOT^*)$.
\end{coro}

Having described the typical behavior of these dynamical classes, we shall further specify the problem by fixing an initial vector. 
More precisely, for a fixed non-zero vector $x\in X$, we shall investigate whether $x$ is a hypercyclic vector for a typical operator in the classes considered above.

\begin{prop}\label{fixed-vector-MIX-dense-l^p}
Let $X=\ell_p$, $1<p<\infty$ and $M>1$. 
Then for every non-zero vector $x$ in $X$,
\[
\operatorname{MIX}_M(X,x):=\{T\in \mathcal{L}_M(X) \colon T \text{ is mixing and }x \text{ is a hypercyclic vector}\}
\]
is SOT$^*$-dense in $\mathcal{L}_M(X)$.
\end{prop}

\begin{proof}
    We first claim that
    \[
    \operatorname{MIX}_{<M}(X)=\left\{T\in\mathcal{L}(X) \colon \|T\|<M  \text{ and }T\text{ is mixing}
    \right\}
    \]
    is SOT$^*$-dense in $\mathcal{L}_M(X)$. 
    Indeed, let $U$ be a non-empty SOT$^*$-open subset of $\mathcal{L}_M(X)$, and choose $T_0\in U$. 
    There exist vectors $u_1,\dotsc,u_s\in X$, functionals $f_1,\dotsc,f_t\in X^*$, and $\epsilon>0$ such that 
    \[
        V=\left\{ T\in\mathcal{L}_M(X) \colon \|(T-T_0)u_i\|<\epsilon,\  \|(T^*-T_0^*)f_j\|<\epsilon,\ 1\leq i \leq s,\ 1\leq j\leq t
        \right\}
    \]
    satisfies $T_0\in V\subset U$. 
    Choose $c\in(0,1)$ sufficiently close to $1$ so that
    \[
        (1-c)\|T_0u_i\|<\frac{\epsilon}{2},\ 1\leq i \leq s , \quad \text{ and } \quad (1-c)\|T_0^*f_j\|<\frac{\epsilon}{2},\mathbb{} 1\leq j\leq t.
    \]
    Set $T_1=cT_0$. Then $\|T_1\|\leq c\|T_0\| \leq cM<M$.
    Choose $r$ such that $\max\{1,\|T_1\|\}<r<M$. 
    Consider the non-empty SOT$^*$-open subset of $\mathcal{L}_r(X)$ given by 
    \[
        W=\biggl\{T\in\mathcal{L}_r(X) \colon \|(T-T_1)u_i\|<\frac{\epsilon}{2},\ \|(T^*-T_1^*)f_j\|<\frac{\epsilon}{2},\ 1\leq i \leq s,\ 1\leq j\leq t \biggr\}.
    \]
    By Proposition \ref{MIX-dense-l^p}, $\operatorname{MIX}_r(X)$ is SOT$^*$-dense in $\mathcal{L}_r(X)$.
    So there exists $R\in W\cap \operatorname{MIX}_r(X)$.
    Hence $R$ is mixing and $\|R\|\leq r<M$. 
    Moreover, for each $1\leq i \leq s$, 
    \[
        \|(R-T_0)u_i\|\leq \|(R-T_1)u_i\|+\|(T_1-T_0)u_i\|<\frac{\epsilon}{2}+\frac{\epsilon}{2}=\epsilon.
    \]
    And similarly, for each $1\leq j\leq t$, $ \|(R^*-T_0^*)f_j\|<\epsilon$. 
    Hence $R\in V\subset U$. 
    This proves that $\operatorname{MIX}_{<M}(X)$ is SOT$^*$-dense in $\mathcal{L}_M(X)$.

    We now prove that $\operatorname{MIX}_M(X,x)$ is dense. 
    With $U$ and $V$ defined as above, we consider the smaller neighborhood 
    \[
        V_0=\biggl\{ T\in\mathcal{L}_M(X) \colon \|(T-T_0)u_i\|<\frac{\epsilon}{2},\ \|(T^*-T_0^*)f_j\|<\frac{\epsilon}{2},\ 1\leq i \leq s,\ 1\leq j\leq t
        \biggr\}.
    \]
    Since $\operatorname{MIX}_{<M}(X)$ is SOT$^*$-dense in $\mathcal{L}_M(X)$, there exists $R\in V_0\cap \operatorname{MIX}_{<M}(X)$.
    Thus $R$ is mixing and $\|R\|<M$.
    Let $C=\max \{1,\|u_1\|,\dotsc,\|u_s\|,\|f_1\|,\dotsc,\|f_t\|\}$ and choose $\eta>0$ such that 
    \[
         \eta< \frac{1}{2}\min \biggl\{M-\|R\|,\frac{\epsilon}{C}\biggr\}.
    \]
    Since $x \neq 0$, by the Hahn-Banach theorem, there exists a functional $g\in X^*$ such that $g(x)=1$. 
    Pick $0<\theta<\frac{1}{2}$ with $\frac{2\|R\|\theta}{1-\theta}<\eta$.
    Since $R$ is mixing, it is hypercyclic. 
    Hence $\operatorname{HC}(R)$ is dense in $X$.
    Since $x\neq 0$, there exists $z\in \operatorname{HC}(R)$ such that $\|z-x\|< \frac{\theta}{\|g\|}$.
    Define $W=I+K$, where $K \colon X\to X$ is defined by 
    \[
        K(v)=g(v)(z-x).
    \]
    Then $W(x)=z$ and $\|W-I\|=\|K\|\leq \|g\|\cdot \|z-x\|<\theta<1$.
    Hence $W$ is invertible with $\|W^{-1}\|<\frac{1}{1-\theta}$.
    Define $S=W^{-1}RW$. 
    Since mixing is invariant under similarity, $S$ is mixing.
    And for each $n\geq 0$, $S^nx=W^{-1}R^nWx=W^{-1}R^nz$.
    Since $z\in \operatorname{HC}(R)$ and $W^{-1}$ is a homeomorphism of $X$, $x\in\operatorname{HC}(S)$.
    Moreover, since $S-R=W^{-1}RW-W^{-1}WR=W^{-1}(R(W-I)-(W-I)R)$,
    \[
        \|S-R\|\leq 2\|W^{-1}\|\cdot \|R\| \cdot\|W-I\|<\frac{2\|R\|\theta}{1-\theta}<\eta<\frac{1}{2}(M-\|R\|).
    \]
    Then
    \[
        \|S\|\leq \|R\|+\|S-R\|<\|R\|+\frac{1}{2}(M-\|R\|)<M.
    \]
    Thus $S\in \mathcal{L}_M(X)$. 
    And for each $1\leq i \leq s$ and $1\leq j\leq t$, we have that 
    \[
        \|(S-T_0)u_i\|\leq \|(S-R)u_i\|+\|(R-T_0)u_i\|<\eta C+\frac{\epsilon}{2}<\epsilon,
    \]
    \[
        \|(S^*-T_0^*)f_j\|\leq \|(S^*-R^*)f_j\|+\|(R^*-T_0^*)f_j\|<\eta C+\frac{\epsilon}{2}<\epsilon.
    \]
    Thus $S\in V\subset U$. 
    Then $S\in U\cap \operatorname{MIX}_M(X,x)$. 
    Therefore $\operatorname{MIX}_M(X,x)$ is dense.
\end{proof}

The next result strengthens Proposition \ref{WMIX-dense-l^p} by requiring a fixed non-zero vector to be hypercyclic.

\begin{prop}
Let $X=\ell_p$, $1<p<\infty$ and $M>1$.
Then for every non-zero vector $x$ in $X$,
\[
 \operatorname{WMIX}_M(X,x):=\{T\in \mathcal{L}_M(X) \colon T \text{ is weakly mixing and }x \text{ is a hypercyclic vector}\}
\]
is residual in $(\mathcal{L}_M(X),SOT^*)$.
\end{prop}

\begin{proof}
    By Proposition \ref{WMIX-dense-l^p}, 
    \[
        \operatorname{WMIX}_M(X):=\{T\in \mathcal{L}_M(X) \colon T \text{ is weakly mixing}\}
    \] 
    is residual in $(\mathcal{L}_M(X),SOT^*)$.
    On the other hand, we claim that 
    \[
         \mathcal{H}_x:=\bigl\{T\in\mathcal{L}_M(X)\colon x\in\operatorname{HC}(T)\bigr\}
    \]
    is a dense $G_\delta$ subset of $(\mathcal{L}_M(X), SOT^*)$. 
    Choose a countable basis of non-empty open subsets $(U_m)_{m\geq 1}$ of $X$. 
    Observe that 
    \[
    \mathcal{H}_x=\bigcap_{m\in\mathbb{N}}\bigcup_{n\geq 0}\left\{T\in\mathcal{L}_M(X) \colon 
    T^n x\in U_m
    \right\}. 
    \]
    For every fixed $n\geq 0$, the map $T\to T^nx$ is SOT$^*$-continuous.
    Hence each set $\left\{T\in\mathcal{L}_M(X) \colon 
    T^n x\in U_m
    \right\}$ is SOT$^*$-open, and therefore $\mathcal{H}_x$ is a $G_\delta$ subset of $\mathcal{L}_M(X)$.
    It is clear that $\operatorname{MIX}_M(X,x) \subset \mathcal{H}_x$.
    By Proposition \ref{fixed-vector-MIX-dense-l^p}, we have that $\operatorname{MIX}_M(X,x)$ is SOT$^*$-dense in $\mathcal{L}_M(X)$.
    Then $\mathcal{H}_x$ is SOT$^*$-dense in $\mathcal{L}_M(X)$.
    Therefore, $\operatorname{WMIX}_M(X,x)$, being the intersection of finitely many residual subsets, is residual. 
    This completes the proof.
\end{proof}

\begin{rem}\label{section3-c0-l1}
    The proofs in this section rely on the finite-dimensional approximation criteria given in Lemmas \ref{dense criteria-lp} and \ref{dense criteria-c0}, as well as the $G_\delta$ characterization corresponding to the dynamical classes under consideration.
    Hence the density and $G_\delta$ arguments in this section extend to $X=c_0$ with respect to SOT$^*$.
    For $X=\ell_1$, analogous statements hold with SOT in place of SOT$^*$.
\end{rem}

\section{Typical weak disjointness and disjoint hypercyclicity}

In this section, we prove assertions (3) and (4) of Theorem \ref{main-result-1}.
We focus on the genericity of weak disjointness and disjoint hypercyclicity. 
The former is characterized by the hypercyclicity of the direct sum $T\oplus S$, while the latter requires this direct sum to admit a hypercyclic vector on the diagonal of $X^2$. 
We investigate these properties first with one operator fixed and then for pairs of operators, and also consider prescribed-vector variants in which a fixed diagonal vector is required to be hypercyclic.
We also establish the typical dense d-hypercyclicity of the $k$-tuple $(T,T^2,\dotsc,T^k)$.

\begin{prop}\label{WD_M(X)-residual-lp}
    Let $X=\ell_p$, $1<p<\infty$, $M>1$ and $S$ be a hypercyclic operator on a Banach space $Y$. 
    Then 
    \[
        \operatorname{WD}_M^S(X):=\{T\in \mathcal{L}_M(X) \colon T \text{ is weakly disjoint from } S\}
    \]
    is residual in $(\mathcal{L}_M(X),SOT^*)$.
\end{prop}

\begin{proof}
    Since $X$ is a separable, choose a countable basis of non-empty open subsets $(U_m)_{m\geq 1}$ of $X$.
    Since $S$ is hypercyclic on the Banach space $Y$, there exists a countable basis $(V_r)_{r\geq 1}$ of non-empty open subsets of $Y$.
    Thus
    \[
        \operatorname{WD}_M^S(X)=\bigcap_{p,q,s,t\geq 1}\bigcup_{n\in N_S(V_p,V_q)}\{T\in \mathcal{L}_M(X) \colon  T^n(U_s)\cap U_t\neq \emptyset\}.
    \]
    For fixed $n,s,t$, let $D_{n,s,t}=\{T\in \mathcal{L}_M(X) \colon  T^n(U_s)\cap U_t\neq \emptyset\}$. 
    Since $T\to T^nx$ is SOT$^*$-continuous, $D_{n,s,t}$ is SOT$^*$-open. 
    Therefore $\operatorname{WD}_M^S(X)$ is a $G_\delta$ subset of $(\mathcal{L}_M(X), SOT^*)$. 

    Moreover, Proposition \ref{MIX-dense-l^p} and the inclusion $\operatorname{MIX}_M(X)\subset \operatorname{WD}_M^S(X)$ imply that $\operatorname{WD}_M^S(X)$ is dense in $(\mathcal{L}_M(X), SOT^*)$. 
    Therefore $\operatorname{WD}_M^S(X)$ is residual in $(\mathcal{L}_M(X), SOT^*)$.
\end{proof}

The fixed-operator result naturally suggests studying weak disjointness as a global relation in the product operator space.

\begin{prop}\label{WD_M(XxX)-residual-lp}
    Let $X=\ell_p$, $1<p<\infty$ and $M>1$.
    Then 
    \[
        \operatorname{WD}_M^{(2)}(X):=\{(T,S)\in \mathcal{L}_M(X)\times \mathcal{L}_M(X) \colon T \text{ is weakly disjoint from } S\}
    \]
   is residual in $(\mathcal{L}_M(X)\times \mathcal{L}_M(X), SOT^*\times SOT^*)$.
\end{prop}

\begin{proof}
    Since $X$ is a separable Banach space, there exists a countable basis $(U_m)_{m\geq 1}$ in $X$. 
    Then $(U_p\times U_q)_{p,q\geq 1}$ is a countable basis of $X^2$. 
    Hence 
    \[
        \operatorname{WD}_M^{(2)}(X)=\bigcap_{p,q,s,t\geq 1}\bigcup_{n\geq 1}\{(T,S) \colon  (T\oplus S)^n(U_p\times U_q)\cap (U_s\times U_t)\neq \emptyset\}.
    \]
    For fixed $n\geq 1$ and $p,q,s,t\geq 1$, let $D_{n,p,q,s,t}=\{(T,S) \colon  (T\oplus S)^n(U_p\times U_q)\cap (U_s\times U_t)\neq \emptyset\}$. 
    We claim that $D_{n,p,q,s,t}$ is open in $\mathcal{L}_M(X)\times \mathcal{L}_M(X)$, endowed with the product SOT. 
    Fix $(T_0,S_0)\in D_{n,p,q,s,t}$. 
    Then there exists $(x,y)\in U_p\times U_q$ such that $(T_0^nx,S_0^ny)\in U_s\times U_t$. 
    Choose $\epsilon_1, \epsilon_2>0$ so that 
    \[
        B(T_0^nx,\epsilon_1)\subset U_s, \quad 
        B(S_0^ny,\epsilon_2)\subset U_t.
    \]
    For each $R\in \mathcal{L}_M(X)$, 
    \[
        \|R^nx-T_0^nx\|\leq \sum_{k=0}^{n-1}\|R^{n-1-k}\| \|(R-T_0)T_0^k x\| \leq \sum_{k=0}^{n-1}M^{n-1-k}\|(R-T_0)T_0^k x\|.
    \]
    Hence, if 
    \[
        \|(R-T_0)T_0^k x\|<\frac{\epsilon_1}{nM^{n-1}}, \quad  k=0,\dotsc,n-1,
    \]
    then $R^nx\in U_s$. 
    So the set of all such $R$, denoted by $U_S$, is an SOT-neighborhood of $T_0$. 
    Similarly, the set of all $Q\in \mathcal{L}_M(X)$ with 
    \[
        \|(Q-S_0)S_0^k y\|<\frac{\epsilon_2}{nM^{n-1}}, \quad k=0,\dotsc,n-1,
    \]
    denoted by $U_T$, is an SOT-neighborhood of $S_0$ and $Q^ny\in U_t$. 
    Thus for each $(R,Q)\in U_S\times U_T$, 
    \[
        (R\oplus Q)^n(x,y)=(R^nx,Q^ny)\in U_s\times U_t.
    \]
    Since $(x,y)\in U_p\times U_q$, it follows that $(R,Q)\in D_{n,p,q,s,t}$. 
    Therefore $D_{n,p,q,s,t}$ is SOT-open.
    Thus $D_{n,p,q,s,t}$ is SOT$^*$-open. 
    Therefore $\operatorname{WD}_M^{(2)}(X)$ is a $G_\delta$ subset of $\mathcal{L}_M(X)\times \mathcal{L}_M(X)$. 
    
    Moreover, by Proposition \ref{MIX-dense-l^p}, $\operatorname{MIX}_M(X)=\{S\in \mathcal{L}_M(X) \colon S \text{ is mixing}\}$ is dense in $(\mathcal{L}_M(X), SOT^*)$. 
    Since 
    \[
        \operatorname{MIX}_M(X) \subset \operatorname{HC}_M(X):=\{S\in \mathcal{L}_M(X) \colon S \text{ is hypercyclic}\},
    \]
    $\operatorname{HC}_M(X)$ is dense in $(\mathcal{L}_M(X), SOT^*)$. 
    For every non-empty open set $U\times V \subset \mathcal{L}_M(X) \times \mathcal{L}_M(X)$, there exist $S\in U\cap \operatorname{MIX}_M(X)$ and $T\in V\cap \operatorname{HC}_M(X)$, thus $(S,T)\in (U\times V)\cap \operatorname{WD}_M^{(2)}(X)$. 
    Hence $\operatorname{WD}_M^{(2)}(X)$ is dense. 
    Therefore $\operatorname{WD}_M^{(2)}(X)$ is residual in $(\mathcal{L}_M(X) \times \mathcal{L}_M(X), SOT^*\times SOT^*)$.
\end{proof}

Having studied the typical occurrence of weak disjointness, we now turn to the stronger notion of disjoint hypercyclicity. 
Recall that disjoint hypercyclicity requires $\operatorname{HC}(T\oplus S)\cap \Delta_X\neq \emptyset$, where $\Delta_X=\{(x,x)\in X\times X\}$.
Under a suitable compatibility condition, a similarity argument transforms a hypercyclic vector of a direct sum into a diagonal hypercyclic vector. 
The following observation is an immediate consequence of the similarity principle from \cite{S10}.

\begin{lem}\label{d-hypercyclic criteria}
    Let $X$ be a Banach space.  
    If there exist an invertible operator $W\in\mathcal{L}(X)$ and $x\in X$ such that 
    $(x,Wx)\in \operatorname{HC}(T\oplus S)$, then $(WTW^{-1},S)$ is d-hypercyclic. 
\end{lem}

\begin{proof}
 We claim that $Wx\in \operatorname{dHC}(WTW^{-1},S)$. 
 Let $U$ and $V$ be two non-empty open subsets of $X$. 
 Since $W$ is bijective, $W^{-1}U$ is also non-empty and open. 
 Now there is $n\in\mathbb{N}$ with $T^nx\in W^{-1}U$ and $S^nWx\in V$. 
 We compute that $(WTW^{-1})^nWx=WT^nW^{-1}Wx=WT^nx\in U$. 
\end{proof}

It is therefore natural to ask whether disjoint hypercyclicity is only achieved by similarity constructions, or is actually typical in suitable operator spaces.
To apply the preceding similarity argument with a prescribed diagonal vector,
we first need to guarantee the existence of many vectors $z$ close to the prescribed
vector $x$ such that $(z,x)$ is hypercyclic for $R\oplus S$. This is provided by
the following elementary Baire-category observation.

\begin{lem}\label{z,x-residual}
Let $X$ be a separable Banach space and $R,S\in\mathcal{L}(X)$.
Assume that $R$ is mixing and that $x\in \operatorname{HC}(S)$.
Then the set
\[
    E:=\{z\in X\colon (z,x)\in \operatorname{HC}(R\oplus S)\}
\]
is a dense $G_\delta$ subset of $X$.
\end{lem}

\begin{proof}
Let $(U_m)_{m\geq 1}$ be a countable basis of non-empty open subsets of $X$.
Then $(U_i\times U_j)_{i,j\geq 1}$ is a countable basis of non-empty open subsets
of $X\times X$. 
We first observe that
\[
    E=\bigcap_{i,j\geq 1}\bigcup_{n\in N_S(x,U_j)} R^{-n}U_i,
\]
set $E_{i,j}=\bigcup_{n\in N_S(x,U_j)} R^{-n}U_i$.
For each $i,j\geq 1$, the set $E_{i,j}$ is open, since it is a union of open sets
$R^{-n}U_i$. 
Hence $E$ is a $G_\delta$ subset of $X$.

It remains to prove that each $E_{i,j}$ is dense in $X$. 
Let $W$ be a non-empty open subset of $X$.
Since $R$ is mixing, there exists $N\geq 1$ such that
\[
    R^n(W)\cap U_i\neq\emptyset
    \quad\text{for every }n\geq N.
\]
On the other hand, since $x\in \operatorname{HC}(S)$, the return set $N_S(x,U_j)$ is infinite.
Therefore we can choose $n\in N_S(x,U_j)$ with $n\geq N$.
For this $n$, we have $R^n(W)\cap U_i\neq\emptyset$.
Then $W\cap E_{i,j}\neq\emptyset$.
Thus $E_{i,j}$ is dense in $X$ for every $i,j\geq 1$.
Therefore $E$ is a countable intersection of open dense subsets of $X$. Since
$X$ is a Baire space, $E$ is a dense $G_\delta$ subset of $X$.
\end{proof}

We are now ready to prove the prescribed diagonal version.
More precisely, for a fixed operator $S$ with $x\in \operatorname{HC}(S)$, we show that a typical operator $T\in \mathcal{L}(X)$ satisfies the stronger property that the direct sum $T\oplus S$ is hypercyclic with the prescribed diagonal vector $(x,x)$ as a hypercyclic vector.

\begin{prop}
Let $X=\ell_p$, $1<p<\infty$, $M>1$, and $S\in \mathcal{L}(X)$. 
If $x\in \operatorname{HC}(S)$, then 
\[
 \operatorname{DHC}_M^S(X,x):=\{T\in \mathcal{L}_M(X) \colon(x,x)\in \operatorname{HC}(T\oplus S)\}
\]
is residual in $(\mathcal{L}_M(X),SOT^*)$.
\end{prop}

\begin{proof}
    Choose a countable basis of non-empty open subsets $(U_m)_{m\geq 1}$ of $X$. 
    Observe that 
    \[
        \operatorname{DHC}_M^S(X,x)=\bigcap_{j,k\geq 1}\bigcup_{n\in N_S(x,U_j)}\{T\in \mathcal{L}_M(X) \colon T^nx\in U_k\}.
    \]
    For every fixed $n\geq 0$, the map $T\to T^nx$ is SOT$^*$-continuous.
    Hence each set $\{T\in \mathcal{L}_M(X) \colon T^nx\in U_k\}$ is SOT$^*$-open, and therefore $\operatorname{DHC}_M^S(X,x)$ is a $G_\delta$ subset of $\mathcal{L}_M(X)$.

    We next prove density. 
    Let $U$ be a non-empty SOT$^*$-open subset of $\mathcal{L}_M(X)$.
    By the proof of Proposition \ref{fixed-vector-MIX-dense-l^p},
    \[
    \operatorname{MIX}_{<M}(X)=\left\{T\in\mathcal{L}(X) \colon \|T\|<M  \text{ and }T\text{ is mixing}
    \right\}
    \]
    is SOT$^*$-dense in $\mathcal{L}_M(X)$. 
    We may choose $R\in U$ such that $R$ is mixing and $\|R\|<M$. 
    Since $R\in U$ and $U$ is an SOT$^*$-open subset, there exist $u_1,\dotsc,u_m\in X$, $f_1,\dotsc,f_s\in X^*$ and $\epsilon>0$ such that 
    \[
        N=\{Q\in\mathcal{L}_M(X)\colon \|(Q-R)u_i\|<\epsilon,\ \|(Q^*-R^*)f_l\|<\epsilon,\ 1\leq i\leq m,\ 1\leq l\leq s\}
    \]
    satisfies $R\in N\subset U$.
    Set $C=\max\{1,\|u_1\|,\dotsc,\|u_m\|,\|f_1\|,\dotsc,\|f_s\|\}$.
    If $Q\in \mathcal{L}(X)$ satisfies $\|Q-R\|<\frac{\epsilon}{C}$, then 
    \[
        \|(Q-R)u_i\|\leq\|Q-R\|\cdot \|u_i\|<\epsilon, \quad 1\leq i\leq m,
    \] 
    and
    \[
       \|(Q^*-R^*)f_l\|\leq\|Q-R\|\cdot \|f_l\|<\epsilon,\quad  1\leq l\leq s.
    \]
    Moreover, since $\|R\|<M$, we may also impose $\|Q-R\|<M-\|R\|$.
    Therefore, choosing $0<\eta <\min\{M-\|R\|,\frac{\epsilon}{C}\}$, we obtain 
    \[
        \|Q-R\|<\eta \quad\Longrightarrow\quad Q\in U \text{ and } \|Q\|<M.
    \]
    Since $R$ is mixing and $x\in \operatorname{HC}(S)$, by Lemma \ref{z,x-residual}, the set
    \[
        E:=\{z\in X\colon (z,x)\in \operatorname{HC}(R\oplus S)\}
    \]
    is dense in $X$.
    Since $x \neq 0$, there exists a functional $g\in X^*$ such that $g(x)=1$. 
    For $z\in X$, define $A=I+K$, where $K \colon X\to X$ is defined by 
    \[
        K(v)=g(v)(z-x).
    \]
    Then $A(x)=z$ and $\|A-I\|=\|K\|\leq \|g\|\cdot\|z-x\|\to 0$ as $z\to x$.
    Hence for $z$ sufficiently close to $x$, $A$ is invertible and $A^{-1}\to I$, $A^{-1}RA\to R$ with respect to the norm topology. 
    Since $E$ is dense in $X$, we may choose $z\in E$ sufficiently close to $x$ so that $A$ is invertible and $\|A^{-1}RA-R\|<\eta$.
    
    Define $T=A^{-1}RA$. 
    Then $T\in U$ and $\|T\|<M$. 
    Since $A(x)=z$, we have $x=A^{-1}z$.
    Since $z\in E$, $(z,x)\in \operatorname{HC}(R\oplus S)$. 
    That is $(z,A^{-1}z)\in \operatorname{HC}(R\oplus S)$.
    By Lemma \ref{d-hypercyclic criteria}, $(A^{-1}z,A^{-1}z)\in \operatorname{HC}(A^{-1}RA\oplus S)$.
    Therefore, $(x,x)\in \operatorname{HC}(T\oplus S)$.
    Thus $T\in U\cap \operatorname{DHC}_M^S(X,x)$.
    Therefore $\operatorname{DHC}_M^S(X,x)$ is a dense $G_\delta$ subset of $(\mathcal{L}_M(X),SOT^*)$.
\end{proof}

\begin{coro}
    Let $X=\ell_p$, $1<p<\infty$, $M>1$ and $S\in \mathcal{L}_M(X)$ be hypercyclic. 
    Then 
    \[
        \operatorname{DHC}_M^S(X):=\{T\in \mathcal{L}_M(X) \colon (T,S) \text{ is d-hypercyclic}\}
    \]
    is residual in $(\mathcal{L}_M(X),SOT^*)$.
\end{coro}

The preceding proposition treats the case where the second operator is fixed.
We next establish the corresponding product result. 
Unlike the preceding proposition, both operators are now variable.

\begin{prop}
    Let $X=\ell_p$, $1<p<\infty$ and $M>1$. 
    Then for every non-zero vector $x$ in $X$,
    \[
        \mathcal{H}_{x,x}:=\{(T,S)\in \mathcal{L}_M(X)\times \mathcal{L}_M(X) \colon(x,x)\in \operatorname{HC}(T\oplus S)\}
    \]
    is residual in $(\mathcal{L}_M(X)\times \mathcal{L}_M(X), SOT^*\times SOT^*)$.
\end{prop}

\begin{proof}
    Choose a countable basis of non-empty open subsets $(U_j\times U_k)_{j,k\geq 1}$ of $X\times X$. 
    Observe that 
    \[
       \mathcal{H}_{x,x}=\bigcap_{j,k\geq 1}\bigcup_{n\geq 0}\{(T,S)\in \mathcal{L}_M(X)\times \mathcal{L}_M(X) \colon (T^nx,S^nx)\in U_j\times U_k\}.
    \]
    For every fixed $n\geq 0$, the map $(T,S)\to (T^nx,S^nx)$ is product SOT$^*$-continuous.
    Hence each set $\{(T,S)\in \mathcal{L}_M(X)\times \mathcal{L}_M(X) \colon (T^nx,S^nx)\in U_j\times U_k\}$ is SOT$^*$-open, and therefore $\mathcal{H}_{x,x}$ is a $G_\delta$ subset of $\mathcal{L}_M(X)\times \mathcal{L}_M(X)$.

    We next prove density. 
    Let $U_1\times U_2$ be a non-empty SOT$^*$-open subset of $\mathcal{L}_M(X)\times \mathcal{L}_M(X)$.
    Choose $(T_0,S_0)\in U_1\times U_2$. 
    Without loss of generality, we can assume that $\|T_0\|<M$ and $\|S_0\|<M$.
    Choose $r$ satisfying $\max\{1,\|T_0\|,\|S_0\|\}<r<M$.
    By Proposition \ref{WD_M(XxX)-residual-lp}, there exists $(R_1,R_2)\in (U_1\cap \mathcal{L}_r(X))\times (U_2 \cap \mathcal{L}_r(X))$ such that $R_1$ is weakly disjoint from $R_2$ and $\|R_1\|\leq r<M$, $\|R_2\|\leq r<M$.
    Since $x \neq 0$, there exists a functional $g\in X^*$ such that $g(x)=1$. 
    Since $R_1\oplus R_2$ is hypercyclic, we can choose $(z_1,z_2)\in \operatorname{HC}(R_1\oplus R_2)$ close to $(x,x)$. 
    For $i=1,2$, define $A_i=I+K_i$, where $K_i \colon X\to X$ is defined by $K_i(v)=g(v)(z_i-x)$.
    Then $A_i(x)=z_i$ and $\|A_i-I\|=\|K_i\|\leq \|g\|\cdot \|z_i-x\|\to 0$ as $z_i\to x$.
    Hence for $z_i$ sufficiently close to $x$, $A_i$ is invertible and $A_i^{-1}\to I$, $A_i^{-1}R_iA_i\to R_i$ with respect to the norm topology. 
    Define
    \[
        T=A_1^{-1}R_1A_1, \quad S=A_2^{-1}R_2A_2.
    \]
    Choose $(z_1,z_2)$ sufficiently close to $(x,x)$ while also satisfying $\|T-R_1\|$ and $\|S-R_2\|<\eta$ for a suitable $\eta$. 
    Therefore $(T,S)\in U_1\times U_2$ and $\|T\|<M$, $\|S\|<M$.
    For each $n\geq 0$, 
    \begin{align*}
        (T\oplus S)^n(x,x)&=(A_1^{-1}\oplus A_2^{-1})(R_1\oplus R_2)^n(A_1\oplus A_2)(x,x)\\
        &=(A_1^{-1}\oplus A_2^{-1})(R_1\oplus R_2)^n(z_1,z_2).
    \end{align*}
    Since $(z_1,z_2)\in \operatorname{HC}(R_1\oplus R_2)$ and $A_1^{-1}\oplus A_2^{-1}$ is a homeomorphism of $X^2$, $(x,x)\in\operatorname{HC}(T\oplus S)$.
    Then $(T,S)\in (U_1\times U_2)\cap \mathcal{H}_{x,x}$.
    Thus $\mathcal{H}_{x,x}$ is dense in $\mathcal{L}_M(X)\times \mathcal{L}_M(X)$. Therefore $\mathcal{H}_{x,x}$ is residual in $(\mathcal{L}_M(X)\times \mathcal{L}_M(X), SOT^*\times SOT^*)$.
\end{proof}

\begin{coro}
    Let $X=\ell_p$, $1<p<\infty$, and $M>1$. 
    Then 
    \[
        \operatorname{DHC}_M(X):=\{(T,S)\in \mathcal{L}_M(X)\times \mathcal{L}_M(X) \colon (T,S) \text{ is d-hypercyclic}\}
    \]
    is residual in $(\mathcal{L}_M(X)\times \mathcal{L}_M(X), SOT^*\times SOT^*)$.
\end{coro}

The pair case naturally extends to finite tuples. 
We next investigate typical d-hypercyclicity for $(T_1,\dotsc,T_k)$, focusing in particular on tuples of the form $(T,T^2,\dotsc,T^k)$.

\begin{lem}
    Let $X=\ell_p$, $1\leq p<\infty$, $M>1$ and $k\geq 2$. 
    Then there exists $T\in \mathcal{L}_M(X)$ such that $(T,T^2,\dotsc,T^k)$ is d-hypercyclic.
\end{lem}

\begin{proof}
    We shall prove the stronger assertion that $(T,T^2,\dotsc,T^k)$ is densely d-hypercyclic.
    It is enough to show that for any non-empty open sets $U,V_1,V_2,\dotsc,V_k$ of $X$, there exist $x\in U$ and $n\geq 0$ such that $T^nx\in V_1, T^{2n}x\in V_2, \dotsc, T^{kn}x\in V_k$. 
    Choose $1<\lambda<M$ and let $T=\lambda B$, where $B$ is the unilateral backward shift on $\ell_p$, i.e. $B(x_0,x_1,x_2,\dotsc)=(x_1,x_2,\dotsc)$.
    Then $\|T\|=\lambda<M$, so $T\in \mathcal{L}_M(X)$.
    Choose finitely supported vectors $u\in U, v_1\in V_1, v_2\in V_2, \dotsc, v_k\in V_k$. 
    And pick $\epsilon_0,\epsilon_1, \dotsc, \epsilon_k>0$ such that $B(u,\epsilon_0)\subset U, B(v_1,\epsilon_1)\subset V_1, \dotsc, B(v_k,\epsilon_k)\subset V_k$.
    For each finitely supported vector $a$, let $r(a)=\max \{s\geq 0 \colon a_s\neq 0\}$. 
    Since $\lambda>1$, we can choose $n\geq 1$ so large that
    \[
        n>\max\{r(u),r(v_1), \dotsc, r(v_k)\},
    \]
    \[
        \lambda^{-n}(\sum_{t=1}^k\|v_t\|)<\epsilon_0, \quad 
        \lambda^{-n}(\sum_{t=j+1}^k\|v_t\|)<\epsilon_j \quad  (j=1,\dotsc,k),
    \]
    where the sum is interpreted as 0 when $j=k$.
    Let F be the unilateral forward shift, i.e. $F(x_0,x_1,x_2,\dotsc)=(0,x_0,x_1,\dotsc)$.
    Define
    \[
        x=u+\sum_{t=1}^k\lambda^{-tn}F^{tn}v_t.
    \]
    Thus
    \[
        \|x-u\|\leq \sum_{t=1}^k\lambda^{-tn}\|v_t\|<\epsilon_0.
    \]
    Hence $x\in U$. 
    For each fixed $j\in \{1,\dotsc,k\}$, since $T^{jn}=\lambda^{jn}B^{jn}$, 
    \[
        T^{jn}x=\lambda^{jn}B^{jn}u+\sum_{t=1}^k \lambda^{(j-t)n}B^{jn}F^{tn}v_t= v_j+\sum_{t=j+1}^k \lambda^{-(t-j)n}F^{(t-j)n}v_t.
    \]
    Thus
    \[
        \|T^{jn}x-v_j\|\leq \sum_{t=j+1}^k \lambda^{-(t-j)n}\|v_t\|\leq \sum_{t=j+1}^k \lambda^{-n}\|v_t\|<\epsilon_j.
    \]
    So $T^{jn}x\in B(v_j,\epsilon_j) \subset V_j$. 
    Thus $U\cap T^{-n}V_1\cap \dotsb \cap T^{-kn}V_k\neq \emptyset$.
    Therefore $(T,T^2,\dotsc,T^k)$ is densely d-hypercyclic, and in particular d-hypercyclic.
\end{proof}

Having established the existence of such operators, we now show that they are not exceptional: the corresponding class is residual in $(\mathcal{L}_M(X),SOT^*)$.

\begin{thm}\label{Mul-DHC-residual}
    Let $X=\ell_p$, $1<p<\infty$ and $M>1$. Then
    \[
        \operatorname{Mul-DHC}_M(X):=\left\{T\in \mathcal{L}_M(X)\colon (T,T^2,\dotsc,T^k) \text{ is densely d-hypercyclic }, \forall k\geq 2 \right\}
    \]
    is residual in $(\mathcal{L}_M(X),SOT^*)$.
    
    In particular, for a typical $T\in\mathcal{L}_M(X)$, the $k$-tuple $(T,T^2,\dotsc,T^k)$ is d-hypercyclic for every $k\geq 2$.
\end{thm}

\begin{proof}
    Fix $k\geq 2$ and let $(U_m)_{m\geq 1}$ be a countable basis of non-empty open subsets of $X$.
    Set 
    \[\mathcal{D}_k
    :=\left\{T\in \mathcal{L}_M(X)\colon (T,T^2,\dotsc,T^k) \text{ is densely d-hypercyclic}\right\}.\]
    We shall prove that $\mathcal{D}_k$ is a dense $G_\delta$ subset of $(\mathcal{L}_M(X),SOT^*)$ for every fixed $k\geq 2$.
    We observe that 
    \[
        \mathcal{D}_k=\bigcap_{i_0,i_1,\dotsc,i_k\geq 1}\bigcup_{n\geq 0}\{T\in \mathcal{L}_M(X) \colon  U_{i_0}\cap T^{-n}U_{i_1}\cap \dotsc \cap T^{-kn}U_{i_k}\neq \emptyset\}.
    \]
    For fixed $n,i_0,\dotsc,i_k$, let $D_{n,i_0,i_1,\dotsc,i_k}=\{T\in \mathcal{L}_M(X) \colon  U_{i_0}\cap T^{-n}U_{i_1}\cap \dotsc T^{-kn}U_{i_k}\neq \emptyset\}$. 
    It is not difficult to verify that $D_{n,i_0,i_1,\dotsc,i_k}$ is SOT-open. 
    Thus $D_{n,i_0,i_1,\dotsc,i_k}$ is SOT$^*$-open. 
    Therefore $\mathcal{D}_k$ is a $G_\delta$ subset of $(\mathcal{L}_M(X), SOT^*)$. 

    We next prove that $\mathcal{D}_k$ is SOT$^*$-dense in $\mathcal{L}_M(X)$. 
    Fix $T_0\in \mathcal{L}_M(X)$,  $x_1,\dotsc,x_s\in X$, $f_1,\dotsc,f_t\in X^*$ and $\epsilon>0$. 
    Without loss of generality, we can assume that $\|T_0\|<M$. 
    It is enough to find $T\in \mathcal{L}_M(X)$ such that 
    \[
        \|(T-T_0)x_i\|<\epsilon,\  i=1,\dotsc,s, \quad
        \|(T^*-T_0^*)f_j\|<\epsilon,\  j=1,\dotsc,t,
    \]
    and for any non-empty open sets, for simplicity, $U:=U_{m_0}, V_1:=U_{m_1}, V_2:=U_{m_2},\dotsc, V_k:=U_{m_k}\subset X$, there exist $z\in U$ and $n\geq 0$ such that $T^nz\in V_1, T^{2n}z\in V_2, \dotsc, T^{kn}z\in V_k$. 
    Choose $\lambda$ such that $\max\{1,\|T_0\| \}<\lambda<M$.
    And choose finitely supported vectors $y_0\in U, y_1\in V_1, y_2\in V_2,\dotsc, y_k\in V_k$. 
    Pick $m$ sufficiently large so that $y_0,y_1,\dotsc,y_k\in E_m:=\operatorname{span}\{e_0,\dotsc,e_{m-1}\}$. 
    Since $P_m\to I$ with respect to SOT$^*$ on both $\ell_p$ and $\ell_q$, we may choose $m$ sufficiently large so that
    \[
        \|P_m T_0P_mx_i- T_0x_i\|<\frac{\epsilon}{2}, \quad 
        M\|(I-P_m)x_i\|<\frac{\epsilon}{2},\ i=1,\dotsc,s,
    \]
    and
    \[
        \|P_m T_0^*P_mf_j- T_0^*f_j\|<\frac{\epsilon}{3}, \quad 
        M\|(I-P_m)f_j\|<\frac{\epsilon}{3},\ j=1,\dotsc,t.
    \]
    Define $A:=P_m T_0|_{E_m}$. 
    Then $A(E_m)\subset E_m$ and $\|A\|\leq \|T_0\|<\lambda<M$. 
    Since $U,V_1,\dotsc,V_k$ are open sets, choose $\delta_0,\delta_1,\dotsc,\delta_{k-1}$ small enough so that 
    \[
        y_0+\delta_0 u_0^{(1)}\in U,
    \]
    \[
        y_r+\delta_r u_0^{(r+1)}\in V_r, \quad r=1,\dotsc,k-1,
    \]
    whenever $u_0^{(1)}, u_0^{(2)}, \dotsc, u_0^{(k)}$ are standard basis vectors supported outside $E_m$. 
    Choose $n\geq 0$ large, to be specified later. 
    Beyond $E_m$, choose $k$ disjoint chains of standard basis vectors 
    \[
        \mathcal{C}_r=\{u_0^{(r)},u_1^{(r)},\dotsc,u_{n-1}^{(r)}\}, \quad r=1,\dotsc,k.
    \]
    For instance, let 
    \[
        u_l^{(r)}=e_{m+(r-1)n+l}, \quad r=1,\dotsc,k, \quad l=0,\dotsc,n-1.
    \]
    Then 
    \[
        \ell_p=E_m\oplus \operatorname{span}\{\mathcal{C}_1\}\oplus \operatorname{span}\{\mathcal{C}_2\}\oplus \dotsb \oplus \operatorname{span}\{\mathcal{C}_k\}\oplus H,
    \]
    where H is the remaining part.
    Define
    \[
        h_r:=\frac{y_r+\delta_r u_0^{(r+1)}-A^ny_{r-1}}{\delta_{r-1}\lambda^{n-1}}, \quad r=1,\dotsc,k-1,
    \]
    and 
    \[
        h_k:=\frac{y_k-A^ny_{k-1}}{\delta_{k-1}\lambda^{n-1}}.
    \]
    Now define $T$ by 
    \[
        T|_{E_m}=A, 
    \]
    \[
        Tu_l^{(r)}=\lambda u_{l+1}^{(r)},\quad r=1,\dotsc,k,  \quad l=0,\dotsc,n-2.
    \]
    \[
        Tu_{n-1}^{(r)}=h_r,
    \]
    and let $T=0$ on all remaining basis vectors. 
    Let $z:=y_0+\delta_0 u_0^{(1)}$. 
    Then $z\in U$. 
    Moreover, we claim that $T^{rn}z\in V_r$ for $r=1,\dotsc,k$. 
    When $r=1$, we have
    \[
        T^nz=T^n(y_0+\delta_0 u_0^{(1)})=A^ny_0+\delta_0 \lambda^{n-1}h_1=y_1+\delta_1 u_0^{(2)} \in V_1.
    \]
    Assume that, for some $1\leq r\leq k-1$, $T^{(r-1)n}z=y_{r-1}+\delta_{r-1}u_0^{(r)}$.
    Then
    \[
        T^{rn}z=T^n(y_{r-1}+\delta_{r-1}u_0^{(r)})=A^ny_{r-1}+\delta_{r-1}\lambda^{n-1}h_r=y_r+\delta_r u_0^{(r+1)}\in V_r.
    \]
    Hence
    \[
        T^{kn}z=T^n(y_{k-1}+\delta_{k-1}u_0^{(k)})=A^ny_{k-1}+\delta_{k-1}\lambda^{n-1}h_k=y_k\in V_k.
    \]
    Thus $U\cap T^{-n}V_1\cap T^{-2n}V_2\cap \dotsc T^{-kn}V_k \neq \emptyset$. 
    
    We next check that $\|T\|\leq M$. 
    Write $T=B+R$, where $B$ is defined similarly to T, except that the ends of the $k$ chains are set to 0, and $R$ satisfies $R(u_{n-1}^{(r)})=h_r$ for $r=1,\dotsc,k$, and is zero elsewhere. 
    Since $\|B\|=\max \{\|A\|,\lambda\}=\lambda$, we have
    \[
        \|T\|\leq \lambda+\sum_{r=1}^k\|h_r\|.
    \] 
    Since $\|A\|< \lambda$, we have that for each fixed $y_r$,
    \[
        \frac{\|A^ny_r\|}{\lambda^{n-1}}\to 0, \quad r=1,\dotsc,k.
    \]
    Therefore $\sum_{r=1}^k\|h_r\|\to 0 (n\to \infty)$.
    Hence we may choose $n$ sufficiently large such that $\sum_{r=1}^k\|h_r\|\leq M-\lambda$. 
    Thus $\|T\|\leq \|B\|+\|R\|\leq M$. 
    
    It remains to show the approximation of $T_0$ for SOT and SOT$^*$.
    For $i=1,\dotsc,s$, since $x_i=P_mx_i+(I-P_m)x_i$, 
    \begin{align*}
    \|(T-T_0)x_i\|&= 
    \|AP_mx_i-T_0x_i+T(I-P_m)x_i)\| \\
    &\leq  \|(P_m T_0P_mx_i- T_0x_i\|+\|T(I-P_m)x_i\| \\
    &\leq \frac{\epsilon}{2}+\frac{\epsilon}{2}=\epsilon.
    \end{align*}
    Furthermore, for each $g\in P_mX^*$, since $A=P_mT_0|_{E_m}$, $A^*g=P_m T_0^*P_mg$. 
    Moreover, since $(T^*g)x=g(Tx)$, $x\in X$, we have that
    $T^*g=A^*g+\sum_{r=1}^k g(h_r)(u_{n-1}^{(r)})^*$. 
    Hence 
    \[
        \|T^*g-P_m T_0^*P_mg\|\leq \|g\|\sum_{r=1}^k\|h_r\|.
    \]
    Let $C:=\max \{1,\|f_1\|,\dotsc,\|f_t\|\}$.
    Since $\sum_{r=1}^k\|h_r\|\to 0$, we may also require $n$ sufficiently large so that $C \sum_{r=1}^k\|h_r\|<\frac{\epsilon}{3}$. 
    Thus 
    \[
        \|T^* P_mf_j-P_mT_0^*P_mf_j\|<\frac{\epsilon}{3}, \quad j=1,\dotsc,t.
    \]
    Therefore, since $f_j=P_mf_j+(I-P_m)f_j$, 
    \begin{align*}
    \|(T^*-T_0^*)f_j\|&\leq 
    \|(T^*P_mf_j-T_0^*f_j\|+\|T^*(I-P_m)f_j)\| \\
    &\leq \|(T^*P_mf_j-P_m T_0^*P_mf_j\|+ \|P_mT_0^*P_mf_j-T_0^*f_j\|+\|T^*(I-P_m)f_j\| \\
    &\leq \frac{\epsilon}{3}+\frac{\epsilon}{3}+\frac{\epsilon}{3}=\epsilon.
    \end{align*}
    
    Therefore, for every fixed $k\geq 2$, the set $\mathcal{D}_k$ is residual in $(\mathcal{L}_M(X),SOT^*)$. 
    Hence $\operatorname{Mul-DHC}_M(X)=\bigcap_{k\geq 2}\mathcal{D}_k$ is residual in $(\mathcal{L}_M(X),SOT^*)$. 
    Since dense d-hypercyclicity implies d-hypercyclicity, the final assertion follows.
\end{proof}

We next strengthen this typical d-hypercyclicity by prescribing the common vector. 
Thus, for a fixed non-zero vector $x\in X$, we require the diagonal vector $(x,\dotsc,x)$ itself to be hypercyclic for $T\oplus T^2\oplus \dotsb \oplus T^k$ for every $k\geq 2$.

\begin{thm}
    Let $X=\ell_p$, $1<p<\infty$, $M>1$. 
    Then for every non-zero vector $x$ in $X$,
    \[
        \mathcal{DHC}_{x}:=\{T\in \mathcal{L}_M(X) \colon(x,\dotsc,x)\in \operatorname{HC}(T\oplus T^2\oplus \dotsb \oplus T^k), \forall k\geq 2 \}
    \]
    is residual in $(\mathcal{L}_M(X), SOT^*)$.
\end{thm}

\begin{proof}
    Fix $k\geq 2$, and choose a countable basis of non-empty open subsets $(U_j)_{j\geq 1}$ of $X^k$.
    Set
    \[
        \mathcal{DHC}_{x,k}=\{T\in \mathcal{L}_M(X) \colon (x,\dotsc,x)\in \operatorname{HC}(T\oplus T^2\oplus \dotsb \oplus T^k)\}.
    \]
    We observe that 
    \[
        \mathcal{DHC}_{x,k}=\bigcap_{j\geq 1}\bigcup_{n\geq 0}\{T\in \mathcal{L}_M(X) \colon(T^nx,T^{2n}x,\dotsc,T^{kn}x)\in U_j\}.
    \]
    For every fixed $n\geq 0$, the map $T\to (T^nx,T^{2n}x,\dotsc,T^{kn}x)$ is SOT$^*$-continuous.
    Hence each set $\{T\in \mathcal{L}_M(X) \colon(T^nx,T^{2n}x,\dotsc,T^{kn}x)\in U_j\}$ is SOT$^*$-open, and therefore $\mathcal{DHC}_{x,k}$ is a $G_\delta$ subset of $\mathcal{L}_M(X)$.

    We next prove that $\mathcal{DHC}_{x,k}$ is SOT$^*$-dense in $\mathcal{L}_M(X)$.
    Let $U$ be a non-empty SOT$^*$-open subset of $\mathcal{L}_M(X)$. 
    Choose $T_0\in U$. 
    Without loss of generality, we can assume that $\|T_0\|<M$.
    Choose $r$ satisfying $\max\{1,\|T_0\|\}<r<M$.
    By Theorem \ref{Mul-DHC-residual}, there exists $R\in U$ such that $(R,R^2,\dotsc,R^k)$ is densely d-hypercyclic and $\|R\|\leq r<M$. 
    Since $x \neq 0$, there exists a functional $g\in X^*$ such that $g(x)=1$. 
    Since $(R,R^2,\dotsc,R^k)$ is densely d-hypercyclic, we may choose $z\in \operatorname{dHC}(R,R^2,\dotsc, R^k)$ sufficiently close to $x$. 
    Define $A=I+K$, where $K \colon X\to X$ is defined by $K(v)=g(v)(z-x)$.
    Then $A(x)=z$ and $\|A-I\|=\|K\|\leq \|g\|\cdot \|z-x\|\to 0$ as $z\to x$.
    Hence for $z$ sufficiently close to $x$, $A$ is invertible and $A^{-1}\to I$, $A^{-1}RA\to R$ with respect to the norm topology. 
    Define
    \[
        T=A^{-1}RA.
    \]
    Choose $z$ sufficiently close to $x$ while also satisfying $\|T-R\|<\eta$ for a suitable $\eta$. 
    Therefore $T\in U$ and $\|T\|<M$.
    For each $n\geq 0$, 
    \begin{align*}
        (T\oplus T^2\oplus \dotsb \oplus T^k)^n(x,\dotsc,x)&=(A^{-1})^{\oplus k}(R\oplus R^2\oplus \dotsb \oplus R^k)^n A^{\oplus k}(x,\dotsc,x)\\
        &=(A^{-1})^{\oplus k}(R\oplus R^2 \oplus\dotsb \oplus R^k)^n(z,\dotsc,z).
    \end{align*}
    Since $z\in \operatorname{dHC}(R,R^2,\dotsc, R^k)$ and $(A^{-1})^{\oplus k}$ is a homeomorphism of $X^k$, 
    $x\in \operatorname{dHC}(T,T^2,\dotsc, T^k)$.
    Then $T\in U\cap \mathcal{DHC}_{x,k}$.
    Therefore $\mathcal{DHC}_{x,k}$ is residual in $(\mathcal{L}_M(X), SOT^*)$.
    Hence $\mathcal{DHC}_{x}=\bigcap_{k\geq 2}\mathcal{DHC}_{x,k}$ is residual in $(\mathcal{L}_M(X),SOT^*)$. 
\end{proof}

\begin{rem}\label{section4-c0-l1}
    The proofs in this section are based on the same finite-dimensional approximation argument as in the preceding section, together with the $G_\delta$ descriptions of weak disjointness, disjoint hypercyclicity, and their prescribed-vector variants. 
    Consequently, the corresponding results remain valid for $X=c_0$ with respect to SOT$^*$.
    For $X=\ell_1$, the corresponding results hold with SOT in place of SOT$^*$.
\end{rem}

We can now complete the proof of the first main result stated in the introduction.

\begin{proof}[Proof of Theorem \ref{main-result-1}]
    Assertion (1) is exactly Proposition \ref{WMIX-dense-l^p}.
    Assertion (2) follows from Theorem \ref{syndetic-meager-l^p}.
    Assertion (3) is Proposition \ref{WD_M(X)-residual-lp}.
    Finally, assertion (4) is precisely Theorem \ref{Mul-DHC-residual}.
    This proves Theorem \ref{main-result-1}.
\end{proof}

We end this section with the following natural question:

\begin{ques}
For which Banach spaces do the preceding results regarding the typicality of dynamical properties hold?
\end{ques}

%\section{\texorpdfstring{$I+B_w$}{I+Bw} type operators}
\section{Operators of the form \texorpdfstring{$I+B_w$}{I+Bw}}

In this section, we prove Theorem \ref{main-result-2}.
We characterize the typical dynamical properties of a special class of operators, $\mathcal{M}=\{I+B_w\in \mathcal{L}(X) \colon w\in c_0(\mathbb{Z})\}$, endowed with the norm topology.
These include weak mixing, topological ergodicity, mild mixing, weak disjointness and disjoint hypercyclicity of powers. 

Recall that Rodr\'iguez-Mart\'inez in \cite{RA12} established that hypercyclicity of both $I+B_w$ and its adjoint $I+B_w^*$ is a typical property among weights $w\in c_0(\mathbb{Z})$.
The key technical ingredient of the proof of this Proposition is the following lemma. 

\begin{lem}[{\cite[Lemma 2.2]{RA12}}]\label{kerT}
    Let $T$ be a continuous operator on a topological vector space $X$. 
    Assume that $x$ and $y$ belong to $\ker^\dagger T:=\operatorname{span}\Bigl(\bigcup_{n=1}^\infty (T^n(X)\cap \ker T^n)\Bigr)$. Then there exists a sequence $\{x_k\}$ in $X$ such that $x_k\to x$ and $(I+T)^kx_k\to y$ as $k\to \infty$.
\end{lem}

We will now use this tool to show that the set of weights $w\in c_0(\mathbb{Z})$ for which both $I+B_w$ and $(I+B_w)^*$ are weakly mixing is a dense $G_\delta$ subset of $c_0(\mathbb{Z})$.

\begin{thm}\label{FSHC_r-star-dense-Gdelta}
    Let $X=\ell_p(\mathbb{Z})$, $1<p<\infty$, and $\mathcal{M}=\{I+B_w \colon w\in c_0(\mathbb{Z})\}\subset (\mathcal{L} (X),\|\cdot\|)$, endowed with the norm topology. Then
    \[
        \operatorname{WMIX}_{I+B}^*:=\{I+B_w \in \mathcal{M} \colon(I+B_w) \text{ and }(I+B_w)^* \text{ are weakly mixing}\}
    \]
    is a dense $G_\delta$ subset of $\mathcal{M}$.
\end{thm}

\begin{proof}
We first note that, for $w,v\in c_0(\mathbb{Z})$, $\|B_w-B_v\|_{\mathcal{L}(\ell_p)}=\|w-v\|_\infty$.
Hence $w\mapsto I+B_w$ is an isometry from $c_0(\mathbb{Z})$ onto $\mathcal{M}$. In particular, $\mathcal{M}$ is a Baire space.
Let $(U_m)_{m\geq 1}$ be a countable basis of $X^2$. 
For $a,b,n\geq 1$, define
\[
D_{a,b,n}=\left\{ T\in\mathcal{M} \colon
(T\oplus T)^n(U_a)\cap U_b\neq \emptyset \right\}.
\]
For fixed $n$, the map $T\mapsto (T\oplus T)^n$ is continuous in the norm topology. 
Hence each $D_{a,b,n}$ is open in $\mathcal{M}$.
Observe that 
\[
\operatorname{WMIX}_{I+B}:=\{I+B_w \in \mathcal{M}:(I+B_w)\text{ is weakly mixing on }X \}=\bigcap_{a,b\geq 1} \bigcup_{n\geq 1}D_{a,b,n}.
\]
Thus $\operatorname{WMIX}_{I+B}$ is a $G_\delta$ subset of $\mathcal{M}$.

We next show that $\operatorname{WMIX}_{I+B}$ is dense.
Fix $a,b\geq 1$. 
We just need to show that $\bigcup_{n\geq 1}D_{a,b,n}$ is dense in $\mathcal{M}$. 
Let $S_w=I+B_w\in\mathcal{M}$ and let $\epsilon>0$. 
Since $c_{00}(\mathbb{Z})^2$ is dense in $X^2$, choose
\[
x=(x_1,x_2)\in U_a\cap c_{00}(\mathbb{Z})^2, \quad
y=(y_1,y_2)\in U_b\cap c_{00}(\mathbb{Z})^2.
\]
Choose $L\in\mathbb{N}$ such that $(\operatorname{supp}x_i)\cup (\operatorname{supp}y_i)\subset [-L,L]$ for $i=1,2$.
Since $w\in c_0(\mathbb{Z})$, choose $N>L$ such that $\sup_{k<-N}|w_k|<\frac{\epsilon}{2}$.
We now choose $w'\in c_0(\mathbb{Z})$ such that 
\[
    \|w'-w\|_\infty<\epsilon,\quad \text{ and } \quad
    \begin{cases}
    w'_k\neq 0, &k\geq -N,\\
    w'_k=0, &k<-N.
    \end{cases}
\]
We claim that every $z\in c_{00}(\mathbb{Z})$ with $\operatorname{supp}z\subset [-L,L]$ belongs to $\ker^\dagger B_{w'}$.
Indeed, choose $n>L+N+1$. 
For every $j\in[-L,L]$, $B_{w'}^n e_j$ contains a factor $w'_\ell$ with $\ell<-N$. Hence $B_{w'}^n e_j=0$.
Thus $e_j\in\ker B_{w'}^n$. 
Moreover,
\[
B_{w'}^n e_{j+n}=\left(w'_{j+n}w'_{j+n-1}\cdots w'_{j+1}\right)e_j.
\]
Since $j\geq -L>-N$, $j+1,\dotsc,j+n$ are at least $-N$, and therefore all corresponding weights are non-zero. 
Hence $e_j\in B_{w'}^n(X)$.
Then $e_j\in B_{w'}^n(X)\cap \ker B_{w'}^n$.
So $z\in \ker^\dagger B_{w'}$. 
In particular,
\[
x_i,y_i\in \ker^\dagger B_{w'},
\quad i=1,2.
\]
By Lemma \ref{kerT}, for each $i=1,2$, there exists a sequence $(x_{i,k})_{k\geq 1}\subset X$ such that
\[
x_{i,k}\to x_i, \quad 
(I+B_{w'})^k x_{i,k}\to y_i.
\]
Since there are only two coordinates, we may choose a common $k$ sufficiently large so that 
\[
(x_{1,k},x_{2,k})\in U_a, \quad
\bigl((I+B_{w'})^k x_{1,k},(I+B_{w'})^k x_{2,k}\bigr)\in U_b.
\]
Then $I+B_{w'}\in D_{a,b,k}$. 
Hence $\bigcup_{n\geq 1}D_{a,b,n}$ is dense in $\mathcal{M}$.  
Therefore $\operatorname{WMIX}_{I+B}$ is a dense $G_\delta$ subset of $\mathcal{M}$.

Define $U:\ell_q(\mathbb{Z})\to \ell_q(\mathbb{Z})$ by $Ue_n=e_{-n}$ for $n\in \mathbb{Z}$. Then $U^{-1}=U$. 
Since $B_we_n=w_ne_{n-1}$, $B_w^*e_n=\overline{w_{n+1}}e_{n+1}$.
Hence, for $n\in \mathbb{Z}$, 
\[
     UB_w^*Ue_n=UB_w^*e_{-n}=U(\overline{w_{-n+1}}e_{-n+1})=\overline{w_{-n+1}}e_{n-1}.
\]     
Next define $\Gamma:c_0(\mathbb{Z})\to c_0(\mathbb{Z})$ by $(\Gamma w)_n=\overline{w_{1-n}}$ for $n\in \mathbb{Z}$. 
Then $B_{\Gamma w}e_n=(\Gamma w)_ne_{n-1}=\overline{w_{-n+1}}e_{n-1}$.
It follows that $UB_w^*U=B_{\Gamma w}$ on $\ell_q(\mathbb{Z})$.
Then $U(I+B_w)^*U=I+B_{\Gamma w}$.
Since weak mixing is invariant under similarity, $(I+B_w)^*$ is weakly mixing if and only if $I+B_{\Gamma w}$ is  weakly mixing. 
Moreover, $\Gamma$ is an isometric homeomorphism of $c_0(\mathbb{Z})$. Indeed, 
\[
    \|\Gamma w\|_\infty=\sup_{n\in \mathbb{Z}}|w_{1-n}|=\|w\|_\infty,
\]
and $\Gamma^2=I$.
From the proof above, we know that the set 
\[
    \operatorname{WMIX}_{I+B}:=\{I+B_w \in \mathcal{M} \colon I+B_w\text{ is weakly mixing}\}
\]
is a dense $G_\delta$ subset of $\mathcal{M}$. 
Then $\{I+B_w \in \mathcal{M} \colon(I+B_w)^*\text{ is weakly mixing }\}$ is a dense $G_\delta$ set. 
Since $\mathcal{M}$ is a Baire space, $\operatorname{WMIX}_{I+B}^*$ is a dense $G_\delta$ subset of $\mathcal{M}$. 
\end{proof}

Although weak mixing is generic in $\mathcal{M}$, the stronger bounded gap return time property is not. 
We next prove that the set of topologically ergodic operators of the form $I+B_w$ is meager in $\mathcal{M}$.

\begin{prop}\label{I+Bw-TERG-meager}
    Let $X=\ell_p(\mathbb{Z})$, $1< p<\infty$ and $\mathcal{M}=\{I+B_w \colon w\in c_0(\mathbb{Z})\}\subset (\mathcal{L} (X),\|\cdot\|)$, endowed with the norm topology. Then 
    \[
        \operatorname{TERG}_{I+B}:=\{I+B_w \in \mathcal{M} \colon I+B_w \text{ is topologically ergodic}\}
    \]
    is meager in $\mathcal{M}$.
\end{prop}

\begin{proof}
    From the proof of Theorem \ref{syndetic-meager-l^p}, we have that 
    \[
        \{I+B_w \colon I+B_w \text{ is topologically ergodic }\}\subset \mathcal{M}\setminus \{I+B_w \colon(I+B_w)^* \text{ is hypercyclic on }\ell_q(\mathbb{Z})\}.
    \]
    By Theorem \ref{FSHC_r-star-dense-Gdelta}, we have that 
    \[
        \{I+B_w \colon(I+B_w)^* \text{ is hypercyclic on }\ell_q(\mathbb{Z})\}
    \]
    is a dense $G_\delta$ subset of $\mathcal{M}$.
    Hence 
    \[
        \mathcal{M}\setminus \{I+B_w \colon(I+B_w)^* \text{ is hypercyclic }\}=\mathcal{M}\setminus  (\cap_{m=1}^\infty O_m)=\cup_{m=1}^\infty (\mathcal{M}\setminus O_m),
    \]
    where $O_m$ is a dense open set. 
    Then for each $m$, $\mathcal{M}\setminus O_m$ is closed and has empty interior. 
    And then $\mathcal{M}\setminus O_m$ is a nowhere dense set.
    Hence $\mathcal{M}\setminus \{I+B_w \colon(I+B_w)^* \text{ is hypercyclic }\}$ is meager.
    Therefore $\operatorname{TERG}_{I+B}$ is meager in $\mathcal{M}$.
\end{proof}  

\begin{coro}
    Let $X=\ell_p(\mathbb{Z})$, $1< p<\infty$ and $\mathcal{M}=\{I+B_w \colon w\in c_0(\mathbb{Z})\}\subset (\mathcal{L} (X),\|\cdot\|)$, endowed with the norm topology. Then 
    \[
        \{I+B_w \in \mathcal{M} \colon I+B_w \text{ is mildly mixing}\}
    \]
    is meager in $\mathcal{M}$.
\end{coro}

Although topological ergodicity is meager in $\mathcal{M}$, weak disjointness still exhibits a typical behavior. 
We next prove this in two forms: first for weakly disjoint partners of a fixed hypercyclic operator $S$, and then for weakly disjoint pairs in $\mathcal{M}\times \mathcal{M}$.

\begin{prop}\label{I+Bw-weakly disjoint}
    Let $X=\ell_p(\mathbb{Z})$, $1< p<\infty$ and $\mathcal{M}=\{I+B_w \colon w\in c_0(\mathbb{Z})\}\subset (\mathcal{L} (X),\|\cdot\|)$, endowed with the norm topology. 
    Let $S$ be a hypercyclic operator on a Banach space $Y$.
    Then 
    \[
        \operatorname{WD}^S_{I+B}:=\{I+B_w \in \mathcal{M} \colon I+B_w \text{ is weakly disjoint from }S\}
    \]
    is residual in $\mathcal{M}$.
\end{prop}

\begin{proof}
Let $(U_m)_{m\geq 1}$ be a countable basis of $X$. 
Since $S$ is a hypercyclic operator on a Banach space $Y$, there exists a countable basis $(V_r)_{r\geq 1}$ in $Y$.
For $p,q,s,t\geq 1$, define
\[
D_{p,q,s,t}=\bigcup_{n\in N_S(V_p,V_q)}\left\{ T\in\mathcal{M} \colon 
T^n(U_s)\cap U_t\neq \emptyset \right\}.
\]
Since the map $T\to T^nx$ is continuous in the norm topology, each $D_{p,q,s,t}$ is open.
Then 
\[
\operatorname{WD}^S_{I+B}=\bigcap_{p,q,s,t\geq 1}D_{p,q,s,t}
\]
is a $G_\delta$ subset of $\mathcal{M}$.

It remains to prove density. 
It is enough to show $D_{p,q,s,t}$ is dense. 
Let $I+B_w\in \mathcal{M}$ and let $\epsilon>0$. 
Since $c_{00}(\mathbb{Z})$ is dense in $X$, choose $x\in U_s\cap c_{00}(\mathbb{Z})$ and $y\in U_t\cap  c_{00}(\mathbb{Z})$. 
By the proof of Theorem \ref{FSHC_r-star-dense-Gdelta}, we obtain that there exists $w'\in c_0(\mathbb{Z})$ such that $\|(I+B_w)-(I+B_{w'})\|<\epsilon$ and  $x,y\in \ker^\dagger B_{w'}$. 
By Lemma \ref{kerT}, there exists a sequence $(x_k)_{k\geq 1}\subset X$ such that
\[
x_k\to x, \quad 
(I+B_{w'})^k x_k\to y.
\]
Since $U_s$ and $U_t$ are open sets and $x\in U_s, y\in U_t$, there exists $N\geq 0$ such that for each $k\geq N$, 
\[
    x_k\in U_s,\quad (I+B_{w'})^k x_k\in U_t.
\]
Since $S$ is hypercyclic, $N_S(V_p,V_q)$ is infinite. 
Choose $n\in N_S(V_p,V_q)$ and $n\geq N$. 
Then $(I+B_{w'})^n(U_s)\cap U_t\neq \emptyset$. 
Hence $I+B_{w'}\in D_{p,q,s,t}$. 
Thus $D_{p,q,s,t}$ is dense.
Therefore $\operatorname{WD}^S_{I+B}$ is residual in $\mathcal{M}$.
\end{proof}

\begin{prop}
    Let $X=\ell_p(\mathbb{Z})$, $1< p<\infty$ and $\mathcal{M}=\{I+B_w \colon w\in c_0(\mathbb{Z})\}\subset (\mathcal{L} (X),\|\cdot\|)$, endowed with the norm topology. 
    Then 
    \[
        \operatorname{WD}(X\times X)_{I+B}:=\{(S,T) \in \mathcal{M}\times \mathcal{M} \colon S \text{ is weakly disjoint from }T \}
    \]
    is residual in $\mathcal{M}\times \mathcal{M}$.
\end{prop}

\begin{proof}
Let $(U_m)_{m\geq 1}$ be a countable basis of $X$. 
For $p,q,s,t\geq 1$, define
\[
D_{p,q,s,t}=\bigcup_{n\geq 1}\left\{ (S,T)\in\mathcal{M}\times \mathcal{M} \colon 
S^n(U_p)\cap U_q\neq \emptyset \text{ and } T^n(U_s)\cap U_t\neq \emptyset \right\}.
\]
Since the maps $S\to S^nx$ and $T\to T^ny$ are continuous in the norm topology, each $D_{p,q,s,t}$ is open.
Then 
\[
\operatorname{WD}(X\times X)_{I+B}=\bigcap_{p,q,s,t\geq 1}D_{p,q,s,t}
\]
is a $G_\delta$ subset of $\mathcal{M}\times \mathcal{M}$.

It remains to prove density. 
Let $\mathcal{U}$ and $\mathcal{V}$ be two arbitrary non-empty open subsets of $\mathcal{M}$.
We shall prove that $(\mathcal{U}\times \mathcal{V})\cap \operatorname{WD}(X\times X)_{I+B}\neq \emptyset $.
Since 
\[
    \operatorname{HC}(\mathcal{M}):=\{S\in\mathcal{M}:S\text{ is hypercyclic}\}
\]
is residual in $\mathcal{M}$, $\mathcal{V}\cap \operatorname{HC}(\mathcal{M})\neq\emptyset$.
Choose $S\in \mathcal{V}\cap \operatorname{HC}(\mathcal{M})$.
By Proposition \ref{I+Bw-weakly disjoint}, 
\[
        \operatorname{WD}^S_{I+B}:=\{T \in \mathcal{M} \colon T \text{ is weakly disjoint from }S\}
\]
is residual in $\mathcal{M}$.
Therefore, $\mathcal{U}\cap \operatorname{WD}^S_{I+B}\neq \emptyset$.
Choose $T\in \mathcal{U}\cap \operatorname{WD}^S_{I+B}$.
Then $T\oplus S$ is hypercyclic, which means that $(T,S)\in \operatorname{WD}(X\times X)_{I+B}$.
Moreover, by the choice of $T$ and $S$, we have $(T,S)\in
(\mathcal{U}\times\mathcal{V})$.
Hence we conclude that $\operatorname{WD}(X\times X)_{I+B}$ is dense in $\mathcal{M}\times \mathcal{M}$.
\end{proof}

The following interpolation lemma is the main ingredient in proving the typical d-hypercyclicity of powers in $\mathcal{M}$.
It enables us to find vectors whose iterates under $I+T$ approximate any prescribed finite collection of target vectors simultaneously.

\begin{lem}\label{kerT-k-disjoint}
Let $X$ be a normed space and $T\in\mathcal{L}(X)$. 
Assume that there exists a sequence $(u_m)_{m\geq 0}\subset X$ such that $u_0\neq 0$ and
\[
Tu_0=0,\quad Tu_m=u_{m-1},\quad m\geq 1.
\]
Set $E=\operatorname{span}\{u_m\colon m\geq 0\}$. 
Then, for every $k\geq 1$ and every $x_0,x_1,\dotsc,x_k\in E$, there exists a sequence $(z_n)_{n\geq 1}\subset E$ such that
\[
(I+T)^{jn}z_n\to x_j,\quad j=0,1,\dotsc,k, 
\]
as $n\to\infty$. 
Here the assertion for $j=0$ means $z_n\to x_0$.
\end{lem}

\begin{proof}
The vectors $u_0,u_1,\dotsc$ are linearly independent. 
Indeed, suppose that $\sum_{m=0}^{M}a_mu_m=0$ with $a_M\neq 0$. 
Applying $T^M$ to the above equality, we obtain $a_Mu_0=0$, which contradicts $u_0\neq 0$.
Hence $(u_m)_{m\geq 0}$ is a basis of  $E=\operatorname{span}\{u_m\colon m\geq 0\}$.

Let $\mathbb{K}[t]$ denote the space of polynomials over $\mathbb{K}$.
For $m\geq 0$, define the Newton polynomials by
\[
 \binom{t}{0}=1 \quad \text{and} \quad 
 \binom{t}{m}=\frac{t(t-1)\dotsb(t-m+1)}{m!}\quad (m\geq 1).
\]
The Newton polynomials $\binom{t}{m}$, $m\geq 0$, form a basis of
$\mathbb{K}[t]$.
Motivated by Pascal's identity, 
\[
\binom{t+1}{m}=\binom{t}{m}+\binom{t}{m-1},
\]
we identify the basis $(u_m)_{m\geq0}$
with the Newton basis $\binom{t}{m}$, $m\geq 0$ by defining the linear isomorphism
\[
 \Phi:E\longrightarrow \mathbb{K}[t],
 \quad
 \Phi(u_m)=\binom{t}{m}.
\]
Hence
\[
 \Phi((I+T)u_m)(t)=\Phi(u_m+u_{m-1})(t)=\binom{t}{m}+\binom{t}{m-1}=\binom{t+1}{m}=\Phi(u_m)(t+1),
\]
and therefore for every $x\in E$ and $N\geq 0$,
\[
    \Phi((I+T)^Nx)(t)=\Phi(x)(t+N).
\]

Now fix $k\geq 1$ and $x_0,x_1,\dotsc,x_k\in E$.
For each $0\leq j\leq k$, let $Q_j=\Phi(x_j)$.
Pick an integer $d\geq 0$ such that $\max_{0\leq j\leq k}\deg Q_j\leq d$. 
Using the standard Lagrange–Hermite interpolation (see e.g. \cite[Chapter 3]{A89}), we will construct a sequence of polynomials $(P_n)_{n\geq 1}$ such that $P_n(t+jn)\to Q_j(t)$ coefficientwise for each $j=0,\dotsc,k$.

For each $0\leq j\leq k$, let $L_j$ be the Lagrange basis polynomial  associated with the interpolation nodes $0,1,\dotsc,k$, given by
\[
L_j(s)=\prod_{\substack{0\leq r\leq k\\ r\neq j}}\frac{s-r}{j-r}.
\]
These polynomials satisfy
\[
L_j(r)=
\begin{cases}
1,& r=j,\\
0,& r\neq j,
\end{cases}
\quad r=0,1,\dotsc,k.
\]
Define 
\[
  H_j(s)=L_j(s)^{d+1}, \quad 0\leq j\leq k.
\]
Then $H_j(j)=1$, and $H_j$ has a root of order at least $d+1$ at every $r\neq j$. 
For $n\geq 1$, define
\[
P_n(t)=\sum_{j=0}^k H_j\left(\frac tn\right)Q_j(t-jn).
\]
Since $\deg P_n\leq k(d+1)+d=:m_0$, the degrees of the polynomials $P_n$ are bounded independently of $n$.
It follows from the calculation that for fixed $\ell\in\{0,\ldots,k\}$,
\[
P_n(t+\ell n)=\sum_{j=0}^kH_j\left(\ell+\frac tn\right)Q_j(t+(\ell-j)n).
\]
For $j=\ell$, we have
\[
 H_\ell\left(\ell+\frac{t}{n}\right)Q_\ell(t)
 \to Q_\ell(t)
\]
coefficientwise. 
If $j\neq\ell$, then $L_j(\ell)=0$.
Hence, by the factor theorem, $s-\ell$ is a factor of $L_j(s)$, and consequently $(s-\ell)^{d+1}$
is a factor of $H_j(s)=L_j(s)^{d+1}$.
Therefore, there exists a polynomial $R_{j,\ell}\in \mathbb{K}[s]$ such that
\[
H_j(\ell+s)=s^{d+1}R_{j,\ell}(s).
\]
Consequently,
\[
 H_j\left(\ell+\frac{t}{n}\right)
 Q_j(t+(\ell-j)n)
 =
 \frac{1}{n}\,t^{d+1}
 R_{j,\ell}\left(\frac{t}{n}\right)
 \left[n^{-d}Q_j(t+(\ell-j)n)\right].
\]
Since $\deg Q_j\leq d$, the coefficients of
$n^{-d}Q_j(t+(\ell-j)n)$ remain bounded as $n\to\infty$. 
Thus the corresponding term is of order $O(\frac{1}{n})$ coefficientwise.
Hence 
\[
P_n(t+\ell n)\to Q_\ell(t)
\]
coefficientwise.

Since $\deg P_n\leq m_0$, we also have $\deg P_n(\,\cdot+\ell n)\leq m_0$
for every $\ell=0,\dotsc,k$. Thus, for each fixed $\ell$, the polynomials $P_n(\,\cdot+\ell n)$ and $Q_\ell$ all belong to the finite-dimensional space
\[
\mathcal{P}_{m_0}
=
\{P\in\mathbb{K}[t]:\deg P\leq m_0\}.
\] 
Since $\mathcal{P}_{m_0}$ is finite-dimensional, coefficientwise convergence is equivalent to convergence with respect to any norm on $\mathcal{P}_{m_0}$. 
Moreover, the restriction of $\Phi^{-1}$ to $\mathcal{P}_{m_0}$ is continuous.  
Let $z_n=\Phi^{-1}(P_n)$. 
It follows that
\[
\Phi^{-1}\bigl(P_n(\,\cdot+\ell n)\bigr) \to  \Phi^{-1}(Q_\ell)=x_\ell.
\]
Since $\Phi((I+T)^Nx)(t)=\Phi(x)(t+N)$ for every $x\in E$ and $N\geq 0$, taking $N=\ell n$ and $x=z_n$, we obtain 
\[
    \Phi\bigl((I+T)^{\ell n}z_n\bigr)(t)=P_n(t+\ell n).
\]
Therefore, 
\[
(I+T)^{\ell n}z_n=\Phi^{-1}\bigl(P_n(\,\cdot+\ell n)\bigr)\to x_\ell
\]
for every $\ell=0,\dotsc,k$. 
Hence 
\[
(I+T)^{jn}z_n\to x_j, \quad j=0,\dotsc,k.
\]
as $n\to\infty$.
\end{proof}

We now apply the above lemma to the family $\mathcal{M}$ of perturbations of weighted backward shifts and prove the typical d-hypercyclicity of powers.

\begin{thm}\label{Mul-DHC-I+B-residual}
Let $X=\ell_p(\mathbb{Z})$, $1\leq  p<\infty$ and $\mathcal{M}=\{I+B_w \colon w\in c_0(\mathbb{Z})\}\subset (\mathcal{L} (X),\|\cdot\|)$, endowed with the norm topology. 
Then 
\[
    \operatorname{Mul-DHC}(X)_{I+B}:=\left\{T\in \mathcal{M}\colon (T,T^2,\dotsc,T^k) \text{ is densely d-hypercyclic }, \forall k\geq 2 \right\}
\]
is residual in $\mathcal{M}$.

In particular, for a typical $I+B_w\in\mathcal{M}$, the $k$-tuple $(I+B_w,(I+B_w)^2,\dotsc,(I+B_w)^k)$ is d-hypercyclic for every $k\geq 2$.
\end{thm}

\begin{proof}
Fix $k\geq2$ and a countable basis $(U_m)_{m\geq1}$ of non-empty open subsets of $X$. 
For $n\geq 0$ and $i_0,\dotsc,i_k\geq 1$, define
\[
D_{n,i_0,\dotsc,i_k}=\{T\in\mathcal{M}\colon U_{i_0}\cap T^{-n}U_{i_1}\cap\dotsc \cap T^{-kn}U_{i_k}\neq \emptyset\}.
\]
We observe that 
\[
\left\{
 T\in\mathcal{M}:(T,T^2,\dotsc,T^k)
 \text{ is densely d-hypercyclic }\right\}=\bigcap_{i_0,\dotsc,i_k\geq 1}\bigcup_{n\geq 0}D_{n,i_0,\dotsc,i_k}.
\]
Each $D_{n,i_0,\dotsc,i_k}$ is open by continuity of $T\mapsto T^{jn}x$ with respect to the norm topology.

To prove density, fix $I+B_w$ and $\epsilon>0$. 
Fix also $i_0,\dotsc,i_k\geq 1$.
Choose $x_j\in U_{i_j}\cap c_{00}(\mathbb{Z})$ for $j=0,1,\dotsc,k$ and choose $L\geq 1$ such that 
\[
  \operatorname{supp}x_j\subset[-L,L], \quad j=0,\dotsc,k.
\]
Since $w\in c_0(\mathbb Z)$, choose $N>L$ and perturb $w$ to $w'\in c_0(\mathbb{Z})$ such that
\[
\|w'-w\|_\infty<\epsilon,
\quad w'_r=0\ (r<-N),
\quad w'_r\neq 0\ (r\geq -N).
\]
Set $r_0=-N-1$. 
Then $w'_{r_0}=0$.
Define $u_0=e_{r_0}$, 
and for $m\geq 1$, 
\[
u_m=\frac{e_{r_0+m}}{w'_{r_0+1}w'_{r_0+2}\dotsc w'_{r_0+m}}.
\]
The denominator is well-defined since
$r_0+j\geq -N$ for every $j\geq 1$.
Indeed, $B_{w'}u_0=0$, and for $m\geq 1$,
\[
B_{w'}u_m=\frac{w'_{r_0+m}e_{r_0+m-1}}{w'_{r_0+1}\dotsc w'_{r_0+m}}=\frac{e_{r_0+m-1}}{w'_{r_0+1}\cdots w'_{r_0+m-1}}=u_{m-1}.
\]
Moreover, for every basis vector $e_q$ appearing in the supports of
$x_0,\dotsc,x_k$, we have $-L\leq q\leq L$. 
Taking $m=q-r_0=q+N+1$, 
we obtain $m\geq N-L+1>0$.
Hence
\[
e_q
\in
\operatorname{span}\{u_m\colon m\geq 0\},
\]
and consequently
\[
x_0,\dotsc,x_k
\in
\operatorname{span}\{u_m\colon m\geq 0\}.
\]
Applying Lemma \ref{kerT-k-disjoint} with $T=B_{w'}$, there exists a sequence $(z_n)_{n\geq 1}$ such that
\[
z_n\to x_0
\quad\text{and}\quad 
(I+B_{w'})^{jn}z_n\to x_j, \quad j=1,\dotsc,k.
\]
Since the sets $U_{i_0},\dotsc,U_{i_k}$ are open, for all sufficiently
large $n$,
\[
z_n\in U_{i_0}
\quad\text{and}\quad 
(I+B_{w'})^{jn}z_n\in U_{i_j},
\quad j=1,\dotsc,k.
\]
Hence for sufficiently large $n$, $I+B_{w'}\in D_{n,i_0,\dotsc,i_k}$ and $I+B_{w'}$ is arbitrarily close to $I+B_w$. 
Therefore the union of the $D_{n,i_0,\dotsc,i_k}$ is dense. 
The Baire theorem completes the proof.
\end{proof}

We conclude this section with the proof of the second main theorem stated in the introduction.

\begin{proof}[Proof of Theorem \ref{main-result-2}]
    Assertion (1) is exactly Theorem \ref{FSHC_r-star-dense-Gdelta}.
    Assertion (2) follows from Proposition \ref{I+Bw-TERG-meager}.
    Assertion (3) is obtained from Proposition \ref{I+Bw-weakly disjoint}.
    Assertion (4) is Theorem \ref{Mul-DHC-I+B-residual}. 
    Thus all assertions of Theorem \ref{main-result-2} have been proved.
\end{proof}

We end this section with the following natural question:
\begin{ques}
    When $X=\ell_p(\mathbb{Z})$, $1< p<\infty$ and $\mathcal{M}=\{I+B_w \colon w\in c_0(\mathbb{Z})\}\subset (\mathcal{L} (X),\|\cdot\|)$, endowed with the norm topology,  
    is $\{I+B_w \in \mathcal{M} \colon I+B_w \text{ is mixing } \}$ dense in $\mathcal{M}$?
\end{ques}

\noindent \textbf{Acknowledgments}: 
The authors were partially supported by the National Key R\&D Program of China (2024YFA1013601), National Natural Science Foundation of China (12222110) and a grant from the Guangdong Provincial Department of Education (2025KCXTD013).


\begin{thebibliography}{88}

\bibitem{A97}
Ansari, Shamim I. Existence of hypercyclic operators on topological vector spaces. J. Funct. Anal. 148 (1997), no. 2, 384--390.

\bibitem{A89}
Atkinson, Kendall E. An introduction to numerical analysis. Second edition. John Wiley \& Sons, Inc., New York, 1989. 

\bibitem{BM09}
Bayart, Frédéric; Matheron, Étienne. Dynamics of linear operators. Cambridge Tracts in Mathematics, 179. Cambridge University Press, Cambridge, 2009.

\bibitem{P89}
Pedersen, Gert K. Analysis now. Graduate Texts in Mathematics, 118. Springer-Verlag, New York, 1989.

\bibitem{B99}
Bernal-González, Luis. On hypercyclic operators on Banach spaces. Proc. Amer. Math. Soc. 127 (1999), no. 4, 1003--1010

\bibitem{B07} 
Bernal-González, Luis. Disjoint hypercyclic operators. Studia Math. 182 (2007), no. 2, 113--131.

\bibitem{BC03} 
Bès, Juan; Chan, Kit C. Denseness of hypercyclic operators on a Fréchet space. Houston J. Math. 29 (2003), no. 1, 195--206.

\bibitem{BP07} 
Bès, Juan; Peris, Alfredo. Disjointness in hypercyclicity. J. Math. Anal. Appl. 336 (2007), no. 1, 297--315.

\bibitem{BMPP16}
Bès, J.; Menet, Q.; Peris, A.; Puig, Y. Recurrence properties of hypercyclic operators. Math. Ann. 366 (2016), no. 1-2, 545--572.

\bibitem{BMPP19}
Bès, J.; Menet, Q.; Peris, A.; Puig, Y. Strong transitivity properties for operators. J. Differential Equations 266 (2019), no. 2-3, 1313--1337.

\bibitem{BP98}
Bonet, José; Peris, Alfredo. Hypercyclic operators on non-normable Fréchet spaces. J. Funct. Anal. 159 (1998), no. 2, 587--595.

\bibitem{C24} 
Cardeccia, Rodrigo. Disjoint hypercyclicity, Sidon sets and weakly mixing operators. Ergodic Theory Dynam. Systems 44 (2024), no. 5, 1315--1329.

\bibitem{C02}
Chan, Kit C. The density of hypercyclic operators on a Hilbert space. J. Operator Theory 47 (2002), no. 1, 131--143. 

\bibitem{E10}
Eisner, Tanja. A "typical" contraction is unitary. Enseign. Math. (2) 56 (2010), no. 3-4, 403--410.

\bibitem{EM13}
Eisner, Tanja; Mátrai, Tamás. On typical properties of Hilbert space operators. Israel J. Math. 195 (2013), no. 1, 247--281.

\bibitem{FW06}
Furstenberg, Hillel; Weiss, Benjamin. The finite multipliers of infinite ergodic transformations. The structure of attractors in dynamical systems (Proc. Conf., North Dakota State Univ., Fargo, N.D., 1977), pp. 127--132, Lecture Notes in Math., 668, Springer, Berlin-New York, 1978. 

\bibitem{G04}
Glasner, Eli. Classifying dynamical systems by their recurrence properties. Topol. Methods Nonlinear Anal. 24 (2004), no. 1, 21--40.

\bibitem{GM22}
Grivaux, S.; Matheron, É. Local spectral properties of typical contractions on $\ell_p$-spaces. Anal. Math. 48 (2022), no. 3, 755--778.

\bibitem{GMM21}
Grivaux, S.; Matheron, É.; Menet, Q. Linear dynamical systems on Hilbert spaces: typical properties and explicit examples. Mem. Amer. Math. Soc. 269 (2021), no. 1315, v+147.

\bibitem{GMMlp21}
Grivaux, Sophie; Matheron, Étienne; Menet, Quentin. Does a typical $\ell_p$-space contraction have a non-trivial invariant subspace? Trans. Amer. Math. Soc. 374 (2021), no. 10, 7359--7410. 
%S. Grivaux, É. Matheron, Q. Menet, Does a typical $\ell_p$-space contraction have a non-trivial invariant
%subspace?, Trans. Am. Math. Soc. 374 (10) (2021) 7359--7410.

\bibitem{GMM26}
Grivaux, Sophie; Matheron, Étienne; Menet, Quentin.  Generic properties of $\ell_p$-contractions and similar operator
topologies. Israel J. Math. (2026), published online.
%, doi: 10.1007/s11856-026-2901-z.

\bibitem{G11}
Grosse-Erdmann, Karl-G.; Peris Manguillot, Alfredo. Linear chaos. Universitext. Springer, London, 2011. 


\bibitem{HY04}
Huang, Wen; Ye, Xiangdong. Topological complexity, return times and weak disjointness. Ergodic Theory Dynam. Systems 24 (2004), no. 3, 825--846.

\bibitem{LLR25}
%Jian Li and Qijing Liao and Yonghang Ruan. 
Li, Jian; Liao, Qijing; Ruan, Yonghang. 
Weak disjointness of hypercyclic operators. to appear in Ergodic Theory Dynam. Systems. arXiv:2512.08519. 

\bibitem{M13}
Moothathu, T. K. Subrahmonian. Two remarks on frequent hypercyclicity. J. Math. Anal. Appl. 408 (2013), no. 2, 843--845.

\bibitem{P72} 
Peleg, Reuven. Weak disjointness of transformation groups. Proc. Amer. Math. Soc. 33 (1972), 165--170.

\bibitem{RA12}
Rodríguez-Martínez, Alejandro. Residuality of sets of hypercyclic operators. Integral Equations Operator Theory 72 (2012), no. 3, 301--308. 

\bibitem{S95}
Salas, Héctor N. Hypercyclic weighted shifts. Trans. Amer. Math. Soc. 347 (1995), no. 3, 993--1004.

\bibitem{S10}
Shkarin, Stanislav. A short proof of existence of disjoint hypercyclic operators. J. Math. Anal. Appl. 367 (2010), no. 2, 713--715.

\bibitem{S15}
Shkarin, Stanislav. Existence theorems in linear chaos. J. Gen. Lie Theory Appl. 9 (2015), no. S1, Art.

\bibitem{W94}
Wu, Pei Yuan. Sums and products of cyclic operators. Proc. Amer. Math. Soc. 122 (1994), no. 4, 1053--1063.

\end{thebibliography}
\end{document}